\documentclass[a4paper, 11pt]{article}

\usepackage[T1]{fontenc}
\usepackage{amsmath}
\usepackage{amssymb}

\usepackage{amsfonts}
\usepackage{amsthm}
\usepackage{indentfirst}
\usepackage{psfrag}
\usepackage{amscd,stmaryrd,latexsym}
\usepackage[pdftex]{graphicx} 
\newtheorem{thm}{Theorem}[section]
\newtheorem{lem}{Lemma}[section]

\newtheorem{Th}{Theorem}
\newtheorem{Prop}{Proposition}[section]

\theoremstyle{remark}
\newtheorem{es}{Example}[section]

\theoremstyle{definition}

\theoremstyle{remark}
\newtheorem{oss}{Remark}[section]

\newcommand{\be}{\begin{equation}}
\newcommand{\ee}{\end{equation}}
\newcommand{\R}{\mathbb{R}}

\newcommand{\N}{\mathbb{N}}

\newcommand{\C}{\mathbb{C}}
\newcommand{\CP}{\mathbb{C}\mathbb{P}}

\newcommand{\G}{\mathbf{g}}

\newcommand\res{\mathop{\hbox{\vrule height 7pt width .5pt depth 0pt
\vrule height .5pt width 6pt depth 0pt}}\nolimits}

\def\eps{\mathop{\varepsilon}}
\def\Mc{\mathop{\mathcal{M}_{}}}
\def\Gc{\mathop{\mathcal{G}}}

\def\Ac{\mathop{\mathcal{A}_{}}}
\def\Ace{\mathop{\mathcal{A}_{\varepsilon}}}

\def\Acr{\mathop{\mathcal{A}_{\rho}}}
\def\Vc{\mathop{\mathcal{V}_{}}}

\def\pt{\mathop{\left( \Phi^{-1} \right)}}
\def\Sc{\mathop{\mathcal{S}_{}}}
\def\Sce{\mathop{\mathcal{S}_{\varepsilon}}}
\def\Vce{\mathop{\mathcal{V}_{\varepsilon}}}
\def\Scr{\mathop{\mathcal{S}_{\rho}}}

\def\Hc{\mathop{\mathcal{H}}}
\def\La{\Lambda}

\def\lt{\tilde{\lambda}}
\def\ti{\tilde}

\def\la{\lambda}
\def\ka{\kappa}

\def\om{\omega}
\def\P{\Phi}

\def\p{\partial}

\def\vt{\vartheta}

\begin{document}

\title{\textbf{Tangent cones to positive-$(1,1)$ De Rham currents}}
\author{\textit{Costante Bellettini}} 
\date{}
\maketitle

\textbf{Abstract}: \textit{We consider positive-$(1,1)$ De Rham currents in arbitrary almost complex manifolds and prove the uniqueness of the tangent cone at any point where the density does not have a jump with respect to all of its values in a neighbourhood. Without this assumption, counterexamples to the uniqueness of tangent cones can be produced already in $\C^n$, hence our result is optimal. The key idea is an implementation, for currents in an almost complex setting, of the classical blow up of curves in algebraic or symplectic geometry. Unlike the classical approach in $\C^n$, we cannot rely on plurisubharmonic potentials.}

\section{Introduction}
\label{setting}

In many problems from analysis one is naturally led to study possibly non-smooth objects: $W^{1,2}$-harmonic maps between manifolds, volume-minimizing currents and weak solutions to equations are a few important examples. In order to understand the behaviour of the object around a singular point, the first study that is typically done is the \textit{blow-up analysis}. We look at the object inside smaller and smaller balls $B_{r_n}(x)$ centered at the chosen point $x$ and dilate to a reference size (e.g. the unit ball). For any sequence $\{r_n\}$ of radii the rescaled objects converge, up to a subsequence, to what is called \textit{a tangent} (tangent maps, tangent cones...). Of course we ask the question: will we get different tangents by choosing different sequences of radii for the blow-up analysis? If not, then we say that the object under investigation has a \textit{unique tangent} at the chosen point. This uniqueness is a very important regularity property, which has been widely investigated in several problems using different techniques. Without hoping to do justice to the vast literature, we present a short overview (see also the survey \cite{H}).

\medskip

Regarding tangent cones at a point $x$ of a mass-minimizing current it is known that the masses of the rescaled currents converge in a non-increasing fashion towards the so-called density at $x$: the speed of convergence is called \textit{rate of decay} of the mass ratio at $x$. An approach often used to prove uniqueness of the tangent cone at $x$ is to show that this rate of decay is fast enough (see \cite{F} 5.4.3). In \cite{W} B. White proved the uniqueness of tangent cone at all points of a $2$-dimensional mass-minimizing integral cycle by showing, via a comparison method, an epiperimetric inequality, from which the desired decay followed. In \cite{PR} D. Pumberger and T. Rivi{\`e}re proved, also by showing the ``fast decay property'', that at any point of a semi-calibrated integral $2$-cycle the tangent cone is unique.

In other works on (semi-)calibrated $2$-cycles alternative proofs have been given by using techniques of slicing with positive intersection: this is the case of integral pseudo-holomorphic $2$-cycles in dimension $4$ (C. H. Taubes in \cite{Ta}, T. Rivi{\`e}re and G. Tian in \cite{RT1}) and integral Special Legendrian $2$-cycles in dimension $5$ (the author and T. Rivi{\`e}re in \cite{BR}, \cite{B}). 

In \cite{RT2} the uniqueness for pseudo holomorphic integral $2$-dimensional cycles is achieved in arbitrary codimension by means of a lower-epiperimetric inequality. 

In \cite{Simon} L. Simon proved that if a tangent cone to a minimal integral current has multiplicity one and has an isolated singularity, then it is unique. This proof applies to tangents at isolated singular points for harmonic maps  taking values into an analytic manifold and is based on the Lojaciewicz inequality, again leading to a rate of decay (for the energy) which implies the uniqueness. On the other hand, White showed in \cite{W2} that tangent maps at isolated singularities of harmonic maps might fail to be unique if the assumption of analiticity on the target manifold is dropped.

Negative answers to the uniqueness of tangent cones have also been obtained in the case of non-rectifiable mass-minimizing currents: this failure was proved for positive-$(p,p)$ normal cycles in a complex manifold  by C. O. Kiselman in \cite{Kis} and in further works, see \cite{Bl} and \cite{BDM}, where necessary and sufficient conditions on the rate of decay of the mass ratio are given, under which the uniqueness holds (these works are closely related to the issue of tangent maps to plurisubharmonic maps). 

\medskip

The problems described so far are of elliptic type, the use of blow-up techniques goes however much further. For example in \cite{AKLR} the authors address a rectifiability issue for a measure arising in the context of conservation laws for hyperbolic PDEs and employ for the proof a delicate blow-up analysis. Turning our attention to a parabolic problem, the classification of possible singularities arising after finite time for a Mean-Curvature Flow is again built upon a blow-up analysis.

\medskip

In the present work we will be dealing with a a first order elliptic pro\-blem: we address the issue of the uniqueness of blow-ups for positive-$(1,1)$ normal cycles in almost complex manifolds. We present a new technique, which does not require the understanding of the rate of decay. 

We will now describe the setting and the connections to other problems, after which a sketch of the proof will be provided.

\medskip

\textbf{Setting}. Let $(\mathcal{M}, J)$ be a smooth almost complex manifold of dimension $2n+2$ (with $n \in \N^*$), endowed with a non-degenerate $2$-form $\om$ compatible with $J$. If $d \om = 0$ then we have a symplectic form, but we will not need to assume closedness. Let $g$ be the associated Riemannian metric, $g(\cdot, \cdot) := \om(\cdot, J \cdot)$.

\medskip

The form $\om$ is a semi-calibration on $\Mc$ for the metric $g$, i.e. the comass $\|\om\|^*$ is $1$; recall that the comass of $\om$ is defined to be
\[||\om||^*:=  \sup \{\langle \om_x, \xi_x \rangle: x \in \Mc,\, \xi_x \text{ is a unit simple $2$-vector at } x\},\] 
 
where the metric that we are using on $T_x \Mc$ is naturally $g_x$. Then $\|\om\|^*=1$ follows from $\om(\cdot, \cdot) = g(J \cdot, \cdot)$, recalling that $J$ is an orthogonal endomorphism. If $\om$ is closed, then we have a classical calibration, as in $\cite{HL}$.

\medskip

Among the oriented $2$-dimensional planes of the Grassmannians $G(x, T_x \Mc)$, we pick those that, represented as unit simple $2$-vectors, realize the equality $\langle \om_x, \xi_x \rangle = 1$. Define the set $\mathcal{G}(\om)$ of  \textit{$2$-planes calibrated by $\om$} as
\[\mathcal{G}(\om) := \cup_{x \in \Mc} \; {\Gc}_x := \cup_{x \in \Mc} \{\xi_x \in G(x, T_x \Mc): \langle \om_x, \xi_x \rangle = 1 \}.\]

\medskip

Before turning to the main object of these work, let us recall a few facts from Geometric Measure theory.

Currents were first introduced by De Rham as the dual space of smooth and compactly supported differential forms (see \cite{DR}). Some distinguished classes of currents have, since the sixties, played a key role in Geometric Measure Theory (see \cite{FF}, \cite{F}, \cite{Morgan}, \cite{G} or \cite{SimonNotes}).

\medskip

For De Rham currents we have the notions of boundary and mass, which we now recall in the case of interest, i.e. a $2$-dimensional De Rham current $C$ (the case of general dimension is completely analogous). 

The boundary $\p C$ of $C$ is the $1$-dimensional current characterized by its action on an arbitrary compactly supported one-form $\alpha$ as follows:

\be
\nonumber
(\p C)(\alpha) := C(d \alpha) = 0 .
\ee

The mass of $C$ is $$M(C):= \sup \{C(\beta): \beta \text{ compactly supported $2$-form}, ||\beta||^* \leq 1 \}.$$

A De Rham current $C$ such that $M(C)$ and $M(\p C)$ are finite is called a \textit{normal current}. Any current $C$ of finite mass is representable by integration (see \cite{G} pages 125-126), i.e. there exist 

(i) a positive Radon measure $\|C\|$,

(ii) a generalized tangent space $\vec{C}_x \in \La_2 \; (T_x \Mc)$, that is defined for $\|C\|$-a.a. points $x$, is $\|C\|$-measurable and has\footnote{The mass-norm for $2$-vectors is defined in duality with the comass on two-forms. The unit ball for the mass-norm on $\Lambda_2 \R^{2n+2}$ is the convex envelope of unit simple $2$-vectors.} mass-norm $1$,

such that the action of $C$ on any $2$-form $\beta$ with compact support is expressed as follows

\be
\nonumber
C(\beta) = \int_{\Mc} \langle \beta, \vec{C} \rangle d\|C\| .
\ee

A current with zero boundary is shortly called a \textit{cycle}. We will consider a \textbf{$\om$-positive} normal $2$-cycle $T$. Equivalent notions of $\om$-positiveness (see \cite{HL} or \cite{HL2}) are

\begin{itemize}
 \item $\vec{T} \, \in \text{convex hull of} \Gc(\om) \;\;\;  \|T\|$-a.e.
 
 \item $\langle \om ,\vec{T} \rangle =1 \;\;\;  \|T\|$-a.e.
\end{itemize}

The last condition is clearly equivalent to the important equality

\be
\label{eq:M=action}
T(\om) = \int_{\Mc} \langle \om, \vec{T} \rangle d\|T\| = M(T).
\ee

Remark that for arbitrary currents $M(C):= \sup \{C(\beta): ||\beta||^* \leq 1 \}$ and in general this $\sup$ need not be achieved. Also remark that for currents of finite mass the action can be extended to forms with non-compact support (actually to forms with merely bounded Borel coefficients, see \cite{G} page 127). So $T(\om)$ in (\ref{eq:M=action}) makes sense.

\medskip

In the case when $\om$ is closed, from (\ref{eq:M=action}) one also gets the important fact that a $\om$-positive $T$ is (locally) homologically mass-minimizing (see \cite{HL}). In the case of a non-closed $\om$, the same argument shows that a $\om$-positive cycle $T$ is locally an almost-minimizer of the mass (also called $\la$-minimizer). When the normal cycle is actually rectifiable (see \cite{F} or \cite{G} for definitions) a common term used, instead of $\om$-positive, is $\om$-(semi)calibrated.

\medskip

In the case we are investigating there is a useful equivalent characterization for the fact that a unit simple $2$-vector at $x$ is in $\Gc_{x}$, i.e. it is $\om_x$-calibrated. Indeed, testing on $w_1 \wedge w_2$ such that $w_1$ and $w_2$ are unit orthogonal vectors at $x$ for $g_x$ and recalling that $J$ is an othogonal endomorphism of the tangent space we get

\be
\label{eq:JandOm}
\om_x(w_1 \wedge w_2) = 1 \Leftrightarrow g_x(J_x (w_1),  w_2) =1 \Leftrightarrow J_x (w_1)= w_2.
\ee

Thus a $2$-plane is in $\Gc_x$ if an only if it is $J_x$-invariant or, in other words, if an only if it is $J_x$-holomorphic. 

So an equivalent way to express $\om$-positiveness is that $\|T\|$-a.e. $\vec{T}$ belongs to the convex hull of $J$-holomorphic simple unit $2$-vectors, in particular $\vec{T}$ itself is  $J$-invariant. For this reason $\om$-positive normal cycles are also called positive-$(1,1)$ normal cycles\footnote{We are using the term \textit{dimension} for a current as it is customary in Geometric Measure Theory, i.e. the dimension of a current is the degree of the forms it acts on. Remark however that in the classical works on positive currents and plurisubharmonic functions, e.g. \cite{L} or \cite{Siu}, our $2$-cycle in $\C^{n+1}$ would actually be called a current of bidimension $(1,1)$ and bidegree $(n,n)$.}. Remarkably the $(1,1)$-condition only depends on $J$, so a positive-$(1,1)$ cycle is $\om$-positive for any $J$-compatible couple $(\om, g)$.

\medskip

Positive cycles satisfy an important \textit{almost monotonicity property}: at any point $x_0$ the mass ratio $\frac{M(T \res B_r(x_0))}{\pi r^{2}}$ is an almost-increasing function of $r$, i.e. it can be expressed as a weakly increasing function of $r$ plus an infinitesimal of $r$. The precise statement can be found in section \ref{result}.

\medskip

Monotonicity yields a well-defined limit $$\nu(x_0) := \lim_{r \to 0} \frac{M(T \res B_r(x_0))}{\pi r^{2}} .$$ This is called the (two-dimensional) \textbf{density} of the current $T$ at the point $x_0$ (\textit{Lelong number} in the classical literature, see \cite{L}). The almost monotonicity property also yields that the density is an \underline{upper semi-continuous} function. 

\medskip

Consider a dilation of $T$ around $x_0$ of factor $r$ which, in normal coordinates around $x_0$, is expressed by the push-forward of $T$ under the action of the map  $\displaystyle \frac{x-x_0}{r}$:

\be
\label{eq:defblowup}
(T_{x_0,r} \res B_1 )(\psi):=\left[ \left( \frac{x-x_0}{r}\right)_\ast T \right](\chi_{B_1} \psi) = T \left( \chi_{B_r(x_0)}\left(\frac{x-x_0}{r}\right)^\ast \psi\right).
\ee

The current $T_{x_0,r}$ is positive for the semi-calibration $\om_{x_0,r}:=\frac{1}{r^2}(r|x-x_0|)^*\om$, with respect to the metric $g_{x_0,r}(X,Y):=\frac{1}{r^2} g\left((r|x-x_0|)_* X, (r|x-x_0|)_* Y\right)$. We thus have the equality $M(T_{x_0,r} \res B_1 ) = \frac{M(T \res B_r(x_0))}{r^{2}}$, where the masses are computed respectively with respect to $g_{x_0,r}$ and $g$. 

The fact that $\displaystyle \frac{M(T \res B_r(x_0))}{r^{2}}$ is monotonically almost-decreasing as $r \downarrow 0$ gives that, for $r\leq r_0$ (for a small enough $r_0$), we are dealing with a family of currents $\{T_{x_0,r} \res B_1\}$ that satisfy the hypothesis of Federer-Fleming's compactness theorem (see \cite{G} page 141) with respect to the flat metric (the metrics $g_{x_0,r}$ converge, as $r \to 0$, uniformly to the flat metric $g_0$). 

Thus there exist a sequence $r_n \to 0$ and a boundaryless current $T_\infty$ such that
\[T_{x_0,r_n} \res B_1 \to T_\infty .\] 
This procedure is called the \textit{blow up limit} and the idea goes back to De Giorgi \cite{DG}. Any such limit $T_\infty$ turns out to be a cone (a so called \textbf{tangent cone} to $T$ at $x_0$) with density at the origin the same as the density of $T$ at $x_0$. Moreover $T_\infty$ is $\om_{x_0}$-positive. 

\medskip

The main issue regarding tangent cones is whether the limit $T_\infty$ depends or not on the sequence $r_n \downarrow 0$ yielded by the compactness theorem, i.e. whether $T_\infty$ is \textbf{unique or not}. It is not hard to check that any two sequences $r_n \to 0$ and $\rho_n \to 0$ fulfilling $a \leq \frac{r_n}{\rho_n} \leq b$ for $a,b>0$ must yield the same tangent cone, so non-uniqueness can arise for sequences with different asymptotic behaviours.

\medskip

The fact that a current possesses a unique tangent cone is a symptom of regularity, roughly speaking of regularity at infinitesimal level. It is generally expected that currents minimizing (or almost-minimizing) functionals such as the mass should have fairly good regularity properties. This issues are however hard in general.

\medskip

The uniqueness of tangent cones is known for some particular classes of \underline{integral} currents, namely for mass-minimizing integral cycles of dimension $2$ (\cite{W}) and for general semi-calibrated integral $2$-cycles (\cite{PR}). 

\medskip

Passing more generally to normal currents, things get harder. Many examples of $\om$-positive normal $2$-cycles can be given by taking a family of pseudoholomorphic curves and assigning a positive Radon measure on it (this can be made rigorous). However $\om$-positive normal $2$-cycles need not be necessarily of this form, as the following example shows.

\begin{es}
In $\R^4 \cong \C^2$, with the standard complex structure, consider the unit sphere $S^3$ and the standard contact form $\gamma$ on it. 

The $2$-dimensional current $C_1$ supported in $S^3$ and dual to $\gamma$, i.e. defined by $C_1(\beta) := \int_{S^3} \gamma \wedge \beta \; d {\Hc}^3$, is positive-$(1,1)$ and its boundary is given by $\p C_1 (\alpha) := \int_{S^3} d \gamma \wedge \alpha \; d {\Hc}^3$, i.e. the boundary is the $1$-current given by the uniform Hausdorff measure on $S^3$ and the Reeb vector field.

Now consider the positive-$(1,1)$ cone $C$ with vertex at the origin, obtained by assigning the uniform measure $\frac{1}{4\pi}{\Hc}^2$ on $\CP^1$, i.e. $C$ is obtained by taking the family of holomorphic disks through the origin and endowing it with a unifom measure of total mass $1$. The current $C_2:=C \res (\R^4 \setminus \overline{B^4_1(0)})$ has boundary $\p C_2 = -\p C_1$, therefore $C_1 + C_2$ is a positive-$(1,1)$ cycle.
\end{es}

This construction shows that a $\om$-positive normal $2$-cycle $T$ is not very rigid and it is not true that, restricting for example to a ball $B$, the current $T \res B$ is the unique minimizer for its boundary (which is instead true for integral cycles). This can be interpreted as a \textit{lack of unique continuation} for these currents.

\medskip

This issue reflects into the fact that the uniqueness of tangent cones to $\om$-positive normal $2$-cycles \textbf{fails} in general, already in the case of the complex manifold $(\C^n, J_0)$, where $J_0$ is the standard complex structure: this was proven by Kiselman \cite{Kis}. Further works extended the result to arbitrary dimension and codimension (see \cite{Bl} and \cite{BDM}, where conditions on the rate of convergence of the mass ratio are given, under which uniqueness holds). 

While in the integrable case $(\C^n, J_0)$ positive cycles have been studied quite extensively, there are no results avaliable when the structure $J$ is \textit{almost complex}.

\medskip

In this work we prove the following result:

\begin{thm}
\label{thm:main1}
Given an almost complex $(2n+2)$-dimensional manifold $(\mathcal{M}, J, \om, g)$ as above, let $T$ be a positive-$(1,1)$ normal cycle, or equivalently a $\om$-positive normal $2$-cycle.

Let $x_0$ be a point of positive density $\nu(x_0) >0$ and assume that there is a sequence $x_m \to x_0$ of points $x_m \neq x_0$ all having positive densities $\nu(x_m)$ and such that $\nu(x_m) \to \nu(x_0)$. 

Then the tangent cone at $x_0$ is unique and is given by $\nu(x_0) \llbracket D \rrbracket$ for a certain $J_{x_0}$-invariant disk $D$.
\end{thm}

The notation $\llbracket D \rrbracket$ stands for the current of integration on $D$. Our proof actually yields the stronger result stated in theorem \ref{thm:main2}.

\medskip

In the integrable case $(\C^n, J_0)$, Siu \cite{Siu} proved a beautiful and remarkable regularity theorem, which in our situation states the following: given $c>0$, the set of points of a positive-$(1,1)$ cycle of density $\geq c$ is made of analytic varieties each carrying a positive, real, constant multiplicity. Therefore, in the integrable case, theorem \ref{thm:main1} follows from Siu's result. 

\medskip

In the non-integrable case, on the other hand, there are no regularity results avaliable at the moment. The proofs of Siu's theorem given in the integrable case, see \cite{Siu}, \cite{Kis2}, \cite{L2}, \cite{D}, strongly rely on a connection with a plurisubharmonic potential for the current, which is not avaliable in the almost complex setting.

In addition to the interest for tangent cones themselves, theorems \ref{thm:main1} and \ref{thm:main2} are a first step towards a regularity result analogous to the one in \cite{Siu}, this time in the non-integrable setting (they can be seen as an infinitesimal version of that). The quest for such a regularity result is strongly motivated by several geometric issues, problems where the structure must be perturbated from a complex to almost complex one, in order to ensure some transversality conditions. Some of these are discussed in \cite{DoT}, \cite{RT1}, \cite{Ti}, \cite{Ti2}. We give here an example related to the study of pseudo-holomorphic maps into algebraic varieties, as those analyzed in \cite{RT1}. Indeed, if $u: M^4 \rightarrow \CP^1$ is pseudoholomorphic and weakly approximable as in \cite{RT1}, with $M^4$ a compact closed $4$-dimensional almost-complex manifold, denoting by $\varpi$ the symplectic form on $\CP^1$, then the $2$-current $U$ defined by $U(\beta):=\int_{M^4} u^*\varpi \wedge \beta$ is a positive-$(1,1)$ normal cycle in $M^4$. As explained in \cite{RT1}, the singular set of $u$ is of zero $\Hc^2$-measure and coincides with the set of points where the density of $U$ is $\geq \epsilon$, for a positive $\epsilon$ depending on $M^4$ (this is a so-called $\eps$-regularity result, see \cite{SU}). Then we would be reduced, in order to understand singularities of $u$, to the study of points of density $\geq \epsilon$ of $U$. Knowing that such a set is made of pseudoholomorphic subvarieties, together with the fact that it is $\Hc^2$-null, would imply that the singular set is made of isolated points, the same result achieved in \cite{RT1} with different techniques. 

The strategy might then be applied to other dimensions. Positive-$(1,1)$ cycles, or more generally other calibrated currents, might also serve for other kind of problems, in which $\eps$-regularity results play a role, for example when dealing with some Yang-Mills fields for high dimensional Gauge Theory (see for example the case of anti-self-dual instantons in section 5 of the survey \cite{Ti2}).

\medskip

\textbf{Sketch of the proof}. The key idea for the proof of our result is to realize for our current a sort of ``\textit{algebraic blow up}''. 

This is a well-known construction in Algebraic and Symplectic Geometry, with the name \textit{``blow up''}. To avoid confusion we will call it \textit{algebraic blow up}, since we have already introduced the notion of blow up as limit of dilations, as customary in Geometric Measure Theory. We now briefly recall the \textit{algebraic blow up} in the complex setting (see figure \ref{fig:blowup}).

\medskip

\textit{Algebraic blow up} (or \textit{proper transform}), (see \cite{MS}). Define $\widetilde{\C}^{n+1}$ to be the submanifold of $\CP^n \times \C^{n+1}$ made of the pairs $(\ell, (z_0, ... z_n))$ such that $(z_0, ... z_n) \in \ell$. 

$\widetilde{\C}^{n+1}$ is a complex submanifold and inherits from $\CP^n \times \C^{n+1}$ the standard complex structure, which we denote $I_0$.  The metric $\G_0$ on $\widetilde{\C}^{n+1}$ is inherited from the ambient $\CP^n \times \C^{n+1}$, that is endowed with the product of the Fubini-Study metric on $\CP^n$ and of the flat metric on $\C^{n+1}$. Let $\Phi:\widetilde{\C}^{n+1} \rightarrow \C^{n+1}$ be the projection map $(\ell, (z_0, ... z_n)) \rightarrow (z_0, ... z_n)$. $\Phi$ is holomorphic for the standard complex structures $J_0$ on $\C^{n+1}$ and $I_0$ on $\widetilde{\C}^{n+1}$ and is a diffeomorphism between $\widetilde{\C}^{n+1} \setminus \left( \CP^n \times \{0\} \right)$ and $\C^{n+1} \setminus  \{0\}$. Moreover the inverse image of $\{0\}$ is $\CP^n \times \{0\}$.

$\widetilde{\C}^{n+1}$ is a complex line bundle on $\CP^n$ but we will later view it as an orientable manifold of (real) dimension $2n+2$. The transformation $\Phi^{-1}$ (called \textit{proper transform}) sends the point $0 \neq (z_0, ... z_n) \in \C^{n+1}$ to the point\- $([z_0, ... z_n], (z_0, ... z_n)) \in \widetilde{\C}^{n+1} \subset \CP^n \times \C^{n+1}$. With the almost complex structures $J_0$ and $I_0$, the $J_0$-holomorphic planes through the origin are sent to the fibers of the line bundle, which are $I_0$-holomorphic planes.

\medskip

\textit{Outline of the argument}. We have a positive-$(1,1)$ normal cycle $T$ in $\C^{n+1}$, at the moment with reference to the standard complex structure $J_0$, and we want to to understand the tangent cones at the origin, that we assume to be a point of density $1$. By assumption we have a sequence of points $x_m \to 0$ with densities converging to $1$. Take a subsequence $x_{m_k}$ such that $\frac{x_{m_k}}{|x_{m_k}|} \to y$ for a point $y \in \p B_1$.

We can make sense (section \ref{blowup}) of the proper transform ${\pt}_* T$, although the map $\P^{-1}$ degenerates at the origin, and prove that ${\pt}_* T$ is a positive-$(1,1)$ normal cycle in $(\widetilde{\C}^{n+1}, I_0, \G_0)$.

The densities of points different than the origin are preserved under the proper transform (see the appendix), therefore the current ${\pt}_* T$ has a sequence of points converging to a certain $y_0$ (that lives in $\CP^n \times \{0\} \subset \widetilde{\C}^{n+1}$) and the densities of these points converge to $1$. More precisely $y_0 =H(y)$, where $H:S^{2n+1} \rightarrow \CP^n$ is the Hopf projection.

${\pt}_* T$ is a positive-$(1,1)$ cycle in $(\widetilde{\C}^{n+1}, I_0, \G_0)$, so by upper semi-continuity of the density $y_0$ is also a point of density $\geq 1$.

\medskip

Turning now to a sequence $T_{0, r_n}$ of dilated currents, with a limiting cone $T_\infty$, we can take the proper transforms ${\pt}_* T_{0, r_n}$ and find that all of them share the features just described, with the same $y_0$ (because radial dilations do not affect the fact that there is a sequence of points of density $1$ whose normalizations converge to $y$). But going to the limit we realize that ${\pt}_* T_{0, r_n}$ weakly converge to the proper transform ${\pt}_* T_\infty$, which is also positive-$(1,1)$. 

The mass is continuous under weak convergence of positive (or calibrated) currents, therefore $y_0$ is a point of density $\geq 1$ for ${\pt}_* T_\infty$. This limit, however, is of a very peculiar form, being the transform of a cone. Recall that the fibers of $\widetilde{\C}^{n+1}$ are holomorphic planes coming from holomorphic planes through the origin of $\C^{n+1}$. Since $T_\infty$ is a positive-$(1,1)$ cone, it is made of a weighted family of holomorphic disks through the origin, as described in (\ref{eq:11link}), and the weight is a positive measure. Then ${\pt}_* T_\infty$ is made of a family of fibers of the line bundle $\widetilde{\C}^{n+1}$ with a positive weight. Then the fact that $y_0$ has density $\geq 1$ implies that the whole fiber $L^{y_0}$ at $y_0$ is counted with a weight  $\geq 1$. Transforming back, $T_\infty$ must contain the plane $\P(L^{y_0})$ with a weight $\geq 1$. 

But the density of $T$ at the origin is $1$, so there is no space for anything else and $T_\infty$ \textbf{must be} the disk $\P(L^{y_0})$ with multiplicity $1$. Since we started from an arbitrary sequence $r_n$, the proof is complete, and it is also clear that $H\left(\frac{x_{m}}{|x_{m}|}\right)$ cannot have accumulation points other than $y_0$.

\medskip

In the almost complex setting we need to adapt the algebraic blow up, respecting the almost complex structure. 

In the next section we recall some facts on monotonicity and tangent cones for $\om$-positive cycles and state the stronger theorem \ref{thm:main2}.

In section \ref{tools} we construct suitable coordinates, used in section \ref{blowup} for the almost complex implementation of the algebraic blow up. In section \ref{blowup} we also prove that the proper transform actually yields a current of finite mass and without boundary. The appendix contains two lemmas: pseudo holomorphic maps preserve both the $(1,1)$-condition and the densities. With all this, in section \ref{proof} we conclude the proof.

\medskip

\textbf{Aknowledgments}. I wish to thank Tristan Rivi\`ere, who introduced me to $(1,1)$-currents and stimulated me to work on the subject. 

This work was partially supported by the Swiss Polytechnic Federal Institute Graduate Research Fellowship ETH-01 09-3.

\section{Tangent cones to positive-$(1,1)$ cycles.}
\label{result}

Given an almost complex $(2n+2)$-dimensional manifold $(\mathcal{M}, J, \om, g)$, let $T$ be a $\om$-positive normal $2$-cycle. Tangent cones are a local matter, it suffices then to work in a chart around the point under investigation.

\medskip

One of the key properties of positive currents is the following \textit{almost monotonicity property} for the mass-ratio. The statement here follows from proposition \ref{Prop:monotonicityapp} in the appendix, which is in turn borrowed from \cite{PR}.

\medskip

\begin{Prop}
\label{Prop:monotonicity}
Let $T$ be a $\om$-positive normal cycle in an open and bounded set of $\R^{2n+2}$, endowed with a metric $g$ and a semicalibration $\om$. We assume that $g$ and $\om$ are $L$-Lipschitz for some constant $L>1$ and that $\frac{1}{5}\mathbb{I}\leq g \leq 5\mathbb{I}$, where $\mathbb{I}$ is the identity matrix, representing the flat metric.

Let $B_r(x_0)$ be the ball of radius $r$ around $x_0$ with respect to the metric $g_{x_0}$ and let $M$ be the mass computed with respect to the metric $g$. There exists $r_0>0$ depending only on $L$ such that, for any $x_0$ and for $r \leq r_0$ the mass ratio $\frac{M(T \res B_r(x_0))}{\pi r^{2}}$ is an almost-increasing function in $r$, i.e. $\displaystyle \frac{M(T \res B_r(x_0))}{\pi r^{2}}=R(r) + o_r(1)$ for a function $R$ that is monotonically non-increasing as $r \downarrow 0$ and a function $o_r(1)$ which is infinitesimal of $r$. 

Independently of $x_0$, the perturbation term $o_r(1)$ is bounded in modulus by $C \cdot L \cdot r$, where $C$ is a universal constant. 
\end{Prop}

\medskip

The fact that $r_0$ and $C$ do not depend on the point yield that the density $\nu(x)$ of $T$ is an \underline{upper semi-continuous} function; the proof is rather standard.

\medskip

Another very important consequence of monotonicity is that the mass is \underline{continuous} and not just lower semi-continuous under weak convergence of semicalibrated or positive cycles. Basically this is due to the fact that computing mass for a $\om$-positive cycle amounts to testing it on the form $\om$, as described in (\ref{eq:M=action}); testing on forms is exactly how weak convergence is defined. This fact is of key importance for this work and will be formally proved when needed (see (\ref{eq:densitypassestolimit}) in section \ref{proof}). 

\medskip

Let us now focus on tangent cones. If we perform the blow up procedure around a point of density $0$, then the limiting cone is unique and is the zero-current. So in this situation there is no issue about the uniqueness of the tangent cone.

We are therefore interested in the limiting behaviour around a point $x_0$ of strictly positive density $\nu(x_0)>0$.

\medskip

From \cite{Bl} we know that any normal positive $2$-cone in $\C^{n+1}$ is a positive Radon measure on $\CP^n$. Combining\footnote{As explained in \cite{Kis} and \cite{Bl}, the family of possible tangent cones at a point $x_0$ must be a convex and connected subset of the space of $\om_{x_0}$-positive cones with density $\nu(x_0)$.} this with the fact that a tangent cone $T_{\infty}$ at $x_0$ to a $\om$-positive cycle is $\om_{x_0}$-positive and has density $\nu(x_0)$ at the vertex, we get that $T_{\infty}$ is represented by a Radon measure, with total measure $\nu(x_0)$, on the set of $\om_{x_0}$-calibrated $2$-planes. Precisely, there exists a positive Radon measure $\tau$ on $\CP^n$ such that, denoting by $D^X$ the $J_{x_0}$-holomorphic unit disk in $B^{2n+2}_1(0)$ corresponding to $X \in \CP^n$, the action of $T_{\infty}$ on any two-form $\beta$ is expressed as

\be
\label{eq:11link}
T_{\infty}(\beta) = \int_{\CP^n} \left\{ \int_{D^X} \langle \beta , \vec{D}^X \rangle \; d \mathcal{L} ^{2} \right\} d \tau(X) .
\ee

\medskip

Let $x_0$ be a point of positive density $\nu(x_0) >0$ and assume that there is a sequence $x_m \to x_0$ of points of positive density $\nu(x_m) \geq \ka >0$ for a fixed $\ka>0$. By upper-semicontinuity of $\nu$ it must be $\nu(x_0) \geq \ka$.

Blow up around $x_0$ for the sequence of radii $|x_m-x_0|$: up to a subsequence we get a tangent cone $T_\infty$. What can we immediately say about this cone?

\medskip

With these dilations, the currents $T_{x_0,|x_m-x_0|}$ always have a point $y_m := \frac{x_m-x_0}{|x_m-x_0|}$ on the boundary of $B_1$ with density $\nu(y_m) \geq \ka$. By compactness we can assume $y_m \to y \in \p B_1$. By monotonicity, for any fixed $\delta >0$, localizing to the ball $B_\delta(y)$ we find, using (\ref{eq:M=action}) and recalling from (\ref{eq:defblowup}) that $T_\infty$ and $T_{x_0,r}$ are positive respectively for $\om_{x_0}$ and $\om_{x_0,r}$,

\be
\nonumber
\begin{split}
 M(T_\infty \res B_\delta(y))= T_\infty(\chi_{B_\delta(y)} \om_{x_0}) = \lim_m T_{x_0,|x_m-x_0|}(\chi_{B_\delta(y)} \om_{x_0}) = \\
 \lim_m T_{x_0,|x_m-x_0|}\left[\frac{\chi_{B_\delta(y)}}{|x_m - x_0|^2} \left(|x_m - x_0| (x-x_0) \right)^* \om \right] =\\
 =\lim_m M(T_{x_0,|x_m-x_0|} \res B_\delta(y) ) \geq \ka \pi \delta^2,
\end{split}
\ee

which\footnote{This computation is an instance of the fact that the mass is continuous under weak convergence of positive currents, unlike the general case when it is just lower semi-continuous.} implies that $y$ has density $\nu(y) \geq \ka$. 

Therefore $T_\infty$ ``must contain'' $\ka \llbracket D \rrbracket$, where $D$ is the holomorphic disk through $0$ and $y$; i.e. $T_\infty - \ka \llbracket D \rrbracket$ is a $\om_{x_0}$-positive cone having density $\nu(x_0) - \ka$ at the vertex.

More precisely, what we have just shown the following well-known lemma. In the sequel $H:S^{2n+1} \to \CP^n$ denotes the standard Hopf projection.

\begin{lem}
Let $x_0$ be a point of positive density $\nu(x_0) >0$ and assume that there is a sequence $x_m \to x_0$, $x_m \neq x_0$, of points of positive density $\nu(x_m) \geq \ka >0$ for a fixed $\ka >0$. Let $\{y_\alpha\}_{\alpha \in A}$ be the set of accumulation points on $\CP^n$ for the sequence $y_m := H\left(\frac{x_m-x_0}{|x_m-x_0|}\right)$. Let $D_\alpha$ be the $J_{x_0}$-holomorphic disk in $T_{x_0}\Mc$ containing $0$ and $H^{-1}(y_\alpha)$. Then for every $\alpha \in A$ there is at least a tangent cone to $T$ at $x_0$ of the form $\ka \llbracket D_\alpha \rrbracket + \tilde{T}_\alpha$, for a $\om_{x_0}$-positive cone $\tilde{T}_\alpha$. 
\end{lem}

In other words, each $\ka \llbracket D_\alpha \rrbracket$ ``must appear'' in at least one tangent cone. What about all other (possibly different) tangent cones that we get by choosing different sequences of radii?

The following result shows that \textbf{any} tangent cone to $T$ at $x_0$ ``must contain'' \underline{each} disk $\ka \llbracket D_\alpha \rrbracket$, for all $\alpha \in A$.

\begin{thm}
\label{thm:main2}
Given an almost complex $(2n+2)$-dimensional manifold $(\mathcal{M}, J, \om, g)$, let $T$ be a $\om$-positive normal $2$-cycle.

Let $x_0$ be a point of positive density $\nu(x_0) >0$ and assume that there is a sequence of points $\{x_m\}$ such that $x_m \to x_0$, $x_m \neq x_0$ and the $x_m$ have positive densities satisfying $\liminf_{m \to \infty} \nu(x_m) \geq \ka$ for a fixed $\ka >0$.

Let $\{y_\alpha\}_{\alpha \in A}$ be the set of accumulation points on $\CP^n$ for the sequence $y_m := H\left(\frac{x_m-x_0}{|x_m-x_0|}\right)$. Let $D_\alpha$ be the $J_{x_0}$-holomorphic disk in $T_{x_0}\Mc$ containing $0$ and $H^{-1}(y_\alpha)$.

Then the points $y_\alpha$'s are finitely many and \underline{any} tangent cone $T_\infty$ to $T$ at $x_0$ is such that $T_\infty - \oplus_\alpha \ka \llbracket D_\alpha \rrbracket$, is a $\om_{x_0}$-positive cone.

\end{thm}

\begin{oss}
It follows that the cardinality of the $y_\alpha$'s is bounded by $\left \lfloor \frac{\nu(x_0)}{\ka} \right \rfloor$. In particular, theorem \ref{thm:main1} follows from this result. 
\end{oss}

\section{Pseudo holomorphic polar coordinates}
\label{tools} \label{polar}

$T$ is $\om$-positive $2$-cycle of finite mass in a $(2n+2)$-dimensional almost complex manifold endowed with a compatible metric and form, $(\mathcal{M}, J, \om, g)$; $T$ is shortly called a $(1,1)$-normal cycle.

\medskip

Since tangent cones to $T$ at a point $x_0$ are a local issue it suffices to work in a chart. We can assume straight from the beginning to work in the geodesic ball of radius $2$, in normal coordinates centered at $x_0$; for this purpose it is enough to start with the current $T$ already dilated enough around $x_0$. Always up to a dilation, without loss of generality we can actually start with the following situation. 

$T$ is a $\om$-positive normal cycle in the unit ball $B^{2n+2}_2(0)$, the coordinates are normal, $J$ is the standard complex structure at the origin, $\om$ is the standard symplectic form at the origin, $\|\om - \om_0\|_{C^{2,\nu}(B^{2n+2}_2)}$ and $\|J - J_0\|_{C^{2,\nu}(B^{2n+2}_2)}$ are small enough.

\medskip

The dilations needed for the blow up are expressed by the map $\displaystyle \frac{x}{r}$ for $r>0$ (we are in a normal chart centered at the origin). So in these coordinates we need to look at the family of currents $$T_{0,r}:=\left( \frac{x}{r} \right)_* T.$$

\medskip

It turns out effective, however, to work in coordinates adapted to the almost-complex structure, as we are going to explain in this section.

\medskip

With coordinates $(z_0, ... z_n)$ in $\C^{n+1}$, we use the notation ($\eps$ is a small positive number)

\be
\label{eq:sce}
\tilde{\Sc}_{\eps}:=\{(z_0, z_1, ... z_n) \in B_{1+\eps}^{2n+2} \subset \C^{n+1}: |(z_1, ..., z_n)| < (1+\eps)|z_0|\}.
\ee

We have a canonical identification of $X = [z_0, z_1, ..., z_n] \in \CP^n$ with the $2$-dimensional plane $D^X=\{\zeta(z_0, z_1, ..., z_n): \zeta \in \C \}$, which is complex for the standard structure $J_0$. 

As $X$ ranges in the open ball 

$${\Vc}_{\eps} \subset \CP^n, \;\; {\Vc}_{\eps}:=\{[z_0, z_1, ..., z_n]:  |(z_1, ..., z_n)| < (1+\eps)|z_0|\} ,$$

the planes $D^X$ foliate the sector $\tilde{\Sc}_{\eps}$. We thus canonically get a \textit{polar foliation} of the sector, by means of holomorphic disks.

\medskip

Let the ball (of radius $2$) $B_2^{2n+2} \subset \R^{2n+2}$ be endowed with an almost complex structure $J$.
The same set as in (\ref{eq:sce}), this time thought of as a subset of $(B_2^{2n+2}, J)$, will be denoted by ${\Sc}_{\eps}$. 

We can get a \textit{polar foliation} of the sector ${\Sc}_{0}$, by means of $J$-pseudo holomorphic disks; this is achieved by perturbing the canonical foliation exhibited for $\tilde{\Sce}$. The case $n=1$ is lemma A.2 in the appendix of \cite{RT1}, the proof is however valid for any $n$: here is the statement.

\medskip

\textbf{Existence of a $J$-pseudo holomorphic polar foliation.}
There exists $\alpha_0 >0$ small enough such that, if $\|J - J_0\|_{C^{2,\nu}(B_2^{2n+2})} < \alpha_0$ and $J=J_0$ at the origin, then the following holds.

There exists a diffeomorphism 

\be
\label{eq:Psi}
\Psi:  \tilde{\Sc}_{\eps} \to  (B_2^{2n+2}, J)\;\; ,
\ee

that extends continuously up to the origin, with $\Psi(0)=0$, with the following properties (see top picture of figure \ref{fig:polar}):

\begin{description}
 \item (i) $\Psi$ sends the $2$-disk $D^X \cap \tilde{\Sc}_{\eps}$ represented by $X = [z_0, z_1, ... z_n] \in \CP^n$ to an embedded $J$-pseudo holomorphic disk through $0$ with tangent $D^X$ at the origin;
 
 \item (ii) the image of $\Psi$ contains $\Sc_0= B^{2n+2}_1 \cap \{|(z_1, ..., z_n)| < |z_0|\}$;
 
 \item (iii) $\|\Psi - Id\|_{C^{2,\nu}( {\Sc}_{\eps})} < C_0$, where $C_0$ is a positive constant that can be made as small as wished by assuming $\alpha_0$ small enough. 
\end{description}
 
The collection $\{\Psi\left(D^{Y}\right): Y \in {\Vc}_{\eps}\}$ of these embedded $J$-pseudo holomorphic disks foliates a neighbourhood of the sector ${\Sc}_{0}$; we will call it a \textit{$J$-pseudo holomorphic polar foliation}.

The proof (see \cite{RT1}) also shows that, in order to foliate ${\Sc}_{0}$, the $\eps$ needed in (\ref{eq:Psi}) can be made small by taking $\alpha_0$ small enough.

\medskip

\textbf{Rescale the foliation}. We are now going to use this \textit{polar foliation} to construct coordinates adapted to $J$.

\medskip

The result in \cite{RT1} actually shows that there exists $\alpha_0$ such that for all $\alpha \in [0, \alpha_0]$, if $\|J - J_0\|_{C^{2,\nu}(B_2^{2n+2})} = \alpha$ and $J=J_0$ at the origin, then there is a map $\Psi_\alpha$ yielding a polar foliation with $\|\Psi_\alpha - Id\|_{C^{2,\nu}( {\Sc}_{\eps})} < o_\alpha(1)$ (an infinitesimal of $\alpha$).

\medskip

We make use however only of the result for $\alpha_0$, as we are about to explain. When we dilate the current $T$ in normal coordinates with a factor $r$ and look at the dilated current in the new ball $B_2^{2n+2}$, we find that it is positive-$(1,1)$ for $J_r$, where $J_r:= (\la_r^{-1})^*J$, i.e. $J_r(V):= (\la_r)_* [J\left((\la_r^{-1})_*V\right)]$.

As $r \to 0$ it holds $\|J_r - J_0\|_{C^{2,\nu}(B_2^{2n+2})} \to 0$. Once we have applied the existence result of the $J$-pseudo holomorphic polar foliation to the ball $B_2^{2n+2}$ endowed with $J$ (assuming $\|J-J_0\|_{C^2}<\alpha_0$), then we get a $J_r$-pseudo holomorphic polar foliation of $(B_2, J_r)$ just as follows.

\begin{figure}[h]
\centering
 \includegraphics[width=9cm]{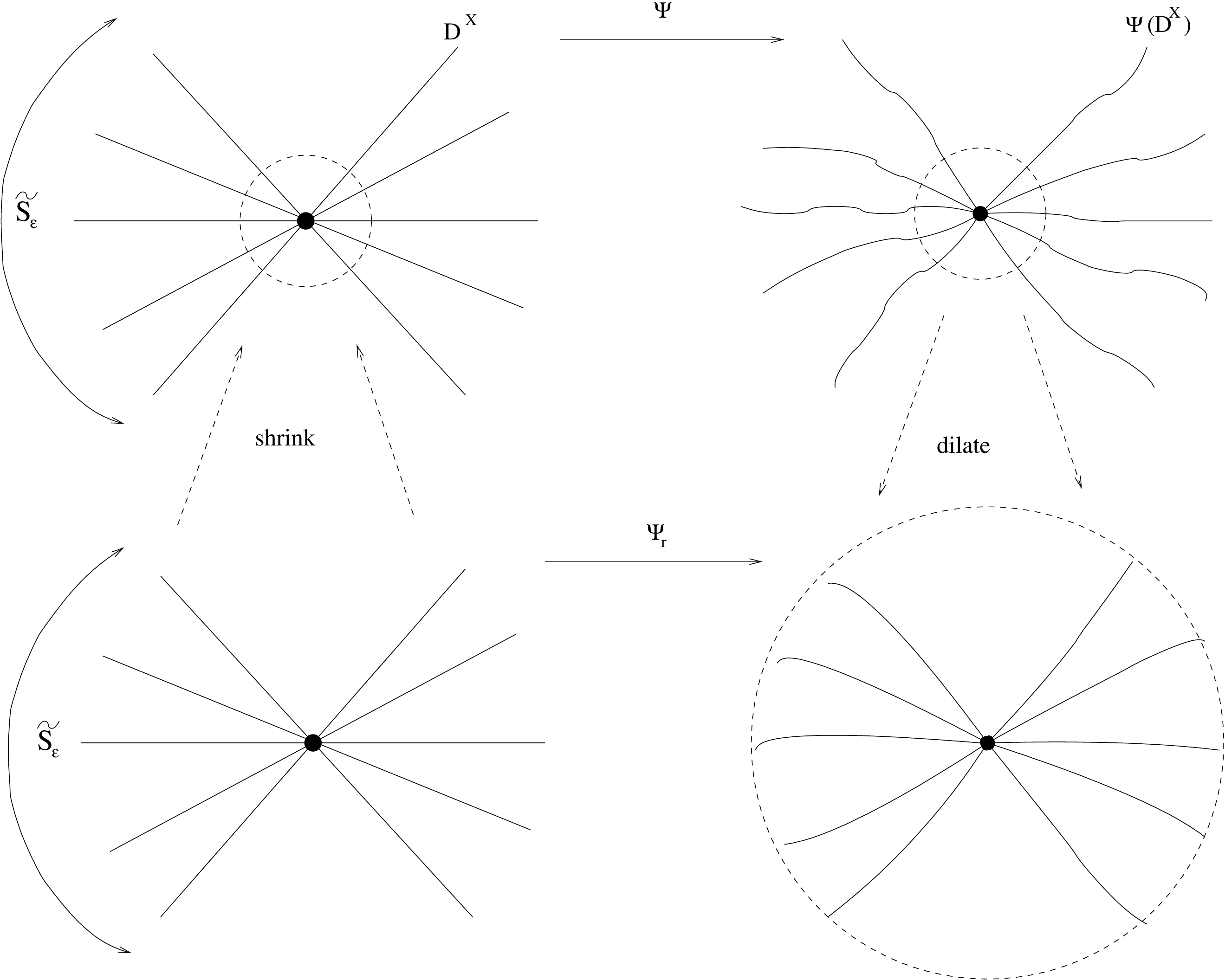}
\caption{$J$-pseudo holomorphic polar foliation via $\Psi$ and $J_r$-pseudo holomorphic polar foliation via $\Psi_r$.}
 \label{fig:polar}
\end{figure}

Let $\lt_r$ be the dilation (in euclidean coordinates) $x \to \frac{x}{r}$; we use the tilda to remind that we are in $\tilde{\Sce}$. The same dilation in normal coordinates in $\Psi({\Sc}_{\eps}) \subset (B_2^{2n+2},J)$ is denoted by $\la_r$. Introduce the map (see figure \ref{fig:polar})

\be
\label{eq:Psir}
\begin{array}{cccc}
\Psi_r: & \tilde{\Sce} & \rightarrow & \left( B^{2n+2}_{2}, J_r \right)\\
 & & \\
 & x & \rightarrow & \la_r \circ \Psi \circ \lt_r^{-1} (x) .
\end{array}
\ee

$\Psi_r$ clearly yields a $J_r$-pseudo holomorphic polar foliation for the ball $B_2^{2n+2}$ endowed with $J_{r}$. Remark, in view of (\ref{eq:equivdilate}), that $\Psi_r$ can actually be defined on the sector $\lt_r(\tilde{\Sce})$.

\medskip

From the proof in \cite{RT1} we get that\footnote{This follows, with reference to the notation in \cite{RT1}, by observing that the map $\Xi_q$ on page 84 (associated to the diffeomorphism that we called $\Psi$) satisfies $\Xi_q \to Id$ uniformly as $q \to 0$, by the condition that above we called (i). Then the $C^{1, \nu}$ bounds there and Ascoli-Arzel\`a's theorem (applied to $D \Psi_r$) yield that $\Psi_r \to Id$ in $C^1$.} $\Psi_r \to Id \text{ in $C^1(\Sce)$ as $r \to 0$}.$









\medskip

\medskip

\textbf{Adapted coordinates}. The aim is to pull back the problem on $\tilde{\Sce}$ via $\Psi$. Endow for this purpose $\tilde{\Sce}$ with the almost complex structure $\Psi^* J$. 

Recall that we have in mind to look at $T_{0,r}$ in $\left( B^{2n+2}_{2}, J_r \right)$ as $r \to 0$. So we are going to study the family

$$\left(\Psi_r ^{-1}\right)_* \left[T_{0,r} \res \left( \Psi_r(\tilde{\Sce})\right)\right]$$

as $r \to 0$. For each $r>0$ these currents are positive-$(1,1)$ normal cycles in $\tilde{\Sce}$ endowed with the almost complex structure $\Psi_r^* J_r$, as proved in lemma \ref{lem:A1}.

It is elementary to check that 
$$\Psi_r^* J_r = (\lt_r^{-1})^* \Psi^* \la_r^* J_r = (\lt_r^{-1})^* \Psi^* J,$$ 
so we can equivalently look, for $r>0$, at $\tilde{\Sce}$ with the almost complex structure $(\lt_r^{-1})^* \Psi^* J$. The latter is obtained from $(\tilde{\Sce}, \Psi^* J)$ by dilation. Remark that $\Psi_r^* J_r \to J_0$ in $C^{0}$ as $r \to 0$; moreover, assuming $\alpha_0$ small enough, the fact that $D \Psi$ is $C^0$-close to $\mathbb{I}$ yields $|\nabla \,(\Psi^* J)| \leq 2 |\nabla J|$.

\medskip

We are looking, in normal coordinates, at a sequence $T_{0,r_n}:=(\la_{r_n})_* T = \left( \frac{x}{r_n} \right)_* T \to T_\infty$. Restricting to $\Psi_{r_n}(\tilde{\Sce})$, i.e. $T_{0,r_n} \res \Psi_{r_n}(\tilde{\Sce})$, we pull back the problem on $\tilde{\Sce}$ and look at 

\be
\label{eq:newcurrents}
\left(\Psi_{r_n}^{-1}\right)_* \left( T_{0,r_n} \res \Psi_{r_n}(\tilde{\Sce}) \right).
\ee

Recalling that $\Psi_r \to Id$ in $C^1$ and that $T_{0,r_n}$ have equibounded masses we have, for any two-form $\beta$,

\be
\label{eq:eqiuvconverg}
\left(\Psi_{r_n}^{-1}\right)_* \left( T_{0,r_n} \res \Psi_{r_n}(\tilde{\Sce}) \right)(\beta) - \left(Id\right)_*\left( T_{0,r_n} \res \Psi_{r_n}(\tilde{\Sce}) \right) (\beta) \to 0.
\ee

This follows with a proof as in step 2 of lemma \ref{lem:A2}, by writing the difference $({\Psi_{r_n}}^{-1})^* \beta - Id^* \beta$ in terms of the coefficients of $\beta$. Then from (\ref{eq:newcurrents}) and (\ref{eq:eqiuvconverg}) we get

\be
\label{eq:eqiuvconverg2}
\lim_{n \to \infty} \left(\Psi_{r_n}^{-1}\right)_* \left( T_{0,r_n} \res \Psi_{r_n}(\tilde{\Sce}) \right) = \left( \lim_{n \to \infty} \left({\la}_{r_n}\right)_* T \right) \res \Sce.
\ee

In the last equality we are identifying the space with the tilda and the one without. On the other hand by (\ref{eq:Psi}) we have 

\be
\label{eq:equivdilate}
\begin{split}
\left(\Psi_{r_n}^{-1}\right)_* \left( T_{0,r_n} \res \Psi_{r_n}(\tilde{\Sce}) \right) = \left[\left(\Psi_{r_n}^{-1}\right)_* {\la_{r_n}}_* \left(T \res  \Psi(\tilde{\Sce})\right)\right] \res \tilde{\Sce} \\
= \left[\left({\lt}_{r_n}\right)_* \left(\Psi^{-1}\right)_* \left(T \res  \Psi(\tilde{\Sce})\right)\right] \res \tilde{\Sce}.
\end{split}
\ee

\medskip

What we have obtained with (\ref{eq:eqiuvconverg2}) and (\ref{eq:equivdilate}) is that, using $\Psi$, we can just pull back $T$ to $\tilde{\Sce}$ endowed with $\Psi^*J$, $\Psi^*g$ and $\Psi^*\om$ and dilate with $\lt_r$ and observe what happens in the limit. All the possible limits of this family are cones, namely all the possible tangent cones to the original $T$, restricted to the sector $\Sce$.

All the information we need about the family $T_{0,r} \res {\Sc}_0$ can be obtained in this way. So we are substituting the blow up in normal coordinates with a different one, that behaves well with respect to $J$ and has the same asymptotic behaviour, i.e. it yields the same cones.

\medskip

Remark that lemmas \ref{lem:A1} and \ref{lem:A2} tell us that $\left(\Psi^{-1}\right)_* \left(T \res  \Psi(\tilde{\Sce})\right)$ is still positive-$(1,1)$ and the densities are preserved. Observe that we cannot use the monotonicity formula for $\left(\Psi^{-1}\right)_* \left(T \res  \Psi(\tilde{\Sce})\right)$ at the origin, since $0$ is now a boundary point. However the monotonicity for $T$ reflects into the following

\begin{lem}
\label{lem:newmonotone}
For the current $\left(\Psi^{-1}\right)_* \left(T \res  \Psi(\tilde{\Sce})\right)$, with respect to the flat metric in $\Sce$, it holds

\be
\label{eq:newmonotone}
\frac{M\left(\left(\Psi^{-1}\right)_* \left(T \res  \Psi(\tilde{\Sce})\right) \res (B_r \cap \tilde{\Sce})\right)}{\pi r^{2}} \leq K
\ee
with a constant $K$ independent of $r$. 
\end{lem}

\begin{proof}[\textbf{proof of lemma \ref{lem:newmonotone}}]
We denote, only for this proof, by $C$ the current $\left(\Psi^{-1}\right)_* \left(T \res  \Psi(\tilde{\Sce})\right)$. Since $|D \Psi-\mathbb{I}| \leq c r^{\nu}$ (where $\mathbb{I} = D(Id)$ is the identity matrix) and $g=g_0 + O(r^2)$ (where $g_0$ is the flat metric), we also get $\Psi^*g = g_0 + O(r^{\nu})$.

Comparing the masses of $C$ with respect to $g_0$ and $\Psi^* g$ we get $$M_{g_0} \left(C  \res (B_r \cap  \tilde{\Sce})\right) \leq (1 + |O(r^{\nu})|)M_{\Psi^* g}\left(C  \res (B_r \cap  \tilde{\Sce})\right),$$

where $B_r$ is always euclidean. Now recall that, by the positiveness of the currents,
 
$$M_{\Psi^* g}\left(C \res (B_r \cap \tilde{\Sce})\right)=\left(C \res (B_r \cap \tilde{\Sce})\right)(\Psi^* \om)= M_g\left( T \res \Psi(B_r \cap \tilde{\Sce} )\right).$$

The condition $|\Psi-Id|\leq c r^{1+\nu}$ implies that $\Psi(B_r \cap \tilde{\Sce} ) \subset B_{r+c r^{1+\nu}} \cap \Sce$. In $\Sce$ coordinates are normal, so, putting all together:

\be
\nonumber
\frac{M_{g_0} \left(C  \res (B_r \cap  \tilde{\Sce})\right)}{r^2} \leq   (1 + |O(r^{\nu})|)\; \frac{(r+c r^{1+\nu})^2 }{r^2} \; \frac{M_{g}\left(T \res B_{r+c r^{1+\nu}}\right)}{(r+c r^{1+\nu})^2 },
\ee

which is equibounded in $r$ by almost monotonicity (proposition \ref{Prop:monotonicity}).

\end{proof}

\medskip

So we restate our problem in the following terms, where we drop the tildas and the pull-backs (resp. push-forwards) via $\Psi$ (resp. $\Psi^{-1}$), since there will be no more confusion arising. 

\medskip

\textbf{New setting: pseudo holomorphic polar coordinates.} 

Endow ${\Sc}_{\eps} \subset B^{2n+2}_2(0)$ with a smooth almost complex structure $J$ such that, denoting by $J_0$ the standard complex structure,

\begin{itemize}
 \item there is $Q>0$ such that for any $0<r<1$, $|J-J_0|_{C^0({\Sc}_{\eps} \cap B_r)} < Q \cdot r$ and $|\nabla J| <Q$ (and $Q$ can be assumed to be small);
 
 \item the $2$-planes $D^X$ (for $X \in {\Vc}_{\eps}$) foliating the sector $\Sce$ are $J$-pseudo holomorphic.
\end{itemize}

Let $\om$ and $g$ be respectively a compatible non-degenerate two-form and the associated Riemannian metric such that $\|\om-\om_0\|_{C^0({\Sc}_{\eps} \cap B_r)} < Q \cdot r$ and $\|g-g_0\|_{C^0({\Sc}_{\eps} \cap B_r)} < Q \cdot r$, where $\om_0$ and $g_0$ are the standard ones. 

Let $T$ be a positive-$(1,1)$ normal cycle in ${\Sc}_{\eps}$.

\medskip

Study the asymptotic behaviour as $r \to 0$ of the family $\left({\la}_{r}\right)_* T$, where $\la_r = \frac{Id}{r}$ in euclidean coordinates. More precisely we can restate theorem \ref{thm:main2} as follows; in theorem \ref{thm:main2} we can assume, up to a rotation and passing to a subsequence, that $y_m=\frac{x_m}{|x_m|} \to (1,0,...,0)$.

\begin{Prop}
 \label{Prop:restated}
With the assumptions just made on $J$ and $T$, assume that there exists a sequence $x_m \to 0$, $0 \neq x_m \in \Sce$, of points all having densities satisfying $\liminf_{m \to \infty} \nu(x_m) \geq \ka$ for a fixed $\ka >0$ and such that $y_m:=\frac{x_m}{|x_m|} \to (1,0,...,0)$. Then any limit
 $$\lim_{r_n \to 0} \left({\la}_{r_n}\right)_* T $$
 is a positive-$(1,1)$ cone (for $J_0$) of the form $\ka \llbracket D^{[1,0,...,0]} \rrbracket + \tilde{T}$, where $\tilde{T}$ is also a positive-$(1,1)$ cone for $J_0$ ($\tilde{T}$ possibly depending on $\{r_n\}$).
\end{Prop}

\begin{oss}
\label{oss:mononewsett}
As observed in (\ref{eq:newmonotone}), our new $T$ satisfies, with respect to the flat metric, $\frac{M(T \res (B_r \cap {\Sc}_{\eps}))}{r^{2}} \leq K$ for a constant independent of $r$.
\end{oss}

\begin{oss}
\label{oss:shrinkeps}

For the proof of proposition \ref{Prop:restated} is suffices to understand the asymptotic behaviour of $T$ in $\Sc_0$, which we will just denote by $\Sc$. So at some point we will look at $T \res \Sc$ and this current has boundary on $\p {\Sc}$. Indeed the operation $\res$ is defined in such a way that it yields a current with support in $\Sc$, but we still view it as a current in the open set $\Sce$.

On the other hand we may wish to look at $T \res \Sc$ as a current in the open set $\Sc$, which means that we only test it against forms compactly supported in $\Sc$: it this case $T$ is boundaryless in $\Sc$. It will be specified when we wish to do so.

\end{oss}

\medskip

\section{Algebraic blow up}
\label{blowup}

The classical symplectic (or algebraic) blow up was recalled in the introduction (maps $\P$ and ${\P}^{-1}$ in figure \ref{fig:blowup}). More details can be found in \cite{MS}. $\widetilde{\C}^{n+1}$ is a complex line bundle over $\CP^n$, that we view as an embedded sumbanifold in $\CP^n \times \C^{n+1}$. We use standard coordinates on $\CP^n \times \C^{n+1}$ coming from the product, so we have $2n$ ``horizontal variables'' and $2n+2$ ``vertical variables''. The standard symplectic form on $\CP^n \times \C^{n+1}$ is given by the two form $\vt_{\CP^n} + \vt_{\C^{n+1}}$, where $\vt_{\CP^n}$ is the standard symplectic form\footnote{In the chart $\C^n \equiv \{z_0 \neq 0\}$ of $\CP^n$, the form $\vt_{\CP^n}$ is expressed, using coordinates $Z=(Z_1, ..., Z_n)$, by $\p \overline{\p} f$, where $f=\frac{i}{2}\log(1 + |Z|^2)$ (see \cite{MS}). The metric $g_{\text{FS}}$ associated to $\vt_{\CP^n}$ and to the standard complex structure is called Fubini-Study metric and it fulfils $\frac{1}{4}\mathbb{I} \leq g_{\text{FS}} \leq 4 \mathbb{I}$ when we compare it to the flat metric on the domain $\{|Z|<1\}$.} on $\CP^n$ extended to $\CP^n \times \C^{n+1}$ (so independent of the ``vertical variables'') and $\vt_{\C^{n+1}}$ is the symplectic two-form on $\C^{n+1}$, extended to $\CP^n \times \C^{n+1}$ (so independent of the ``horizontal variables''). To $\vt_{\CP^n} + \vt_{\C^{n+1}}$ we associate the standard metric, i.e. the product of the Fubini-Study metric on $\CP^n$ and the flat metric on $\C^{n+1}$. The associated complex structure is denoted $I_0$.

\medskip

As a complex submanifold, $\widetilde{\C}^{n+1}$ inherits from the ambient space a complex structure, still denoted $I_0$, and the restricted symplectic form $\vt_0:= \mathcal{E}^*\left(\vt_{\CP^n} + \vt_{\C^{n+1}}\right)$, where $\mathcal{E}$ is the embedding in $\CP^n \times \C^{n+1}$. Let further $\G_0$ denote the ambient metric restricted to $\widetilde{\C}^{n+1}$: $\G_0$ is then compatible with $I_0$ and $\vt_0$, i.e. ${\vt}_0(\cdot, \cdot):= \G_0(\cdot, -I_0 \cdot)$. 





We now turn to the almost complex situation and will adapt the previous construction by building on the results of section \ref{polar}. 

\medskip

\textbf{Implementation in the almost complex setting}. With the notation 

$$\Sce=\{(z_0, z_1, ... z_n) \in B_{1+\eps}^{2n+2} \subset \C^{n+1}: |(z_1, ..., z_n)| < (1+\eps)|z_0|\}$$

for $\eps \geq 0$ as in (\ref{eq:sce}), let ${\Sc}={\Sc}_0$. Also set $\Vce:=\left\{\sum_{j=1}^n \frac{|z_j|^2}{|z_0|^2} <1+\eps \right\} \subset \CP^n$ and $\Vc=\Vc_0$.

\medskip

The inverse image $\Phi^{-1} ({\Sce})$ is given by $\{(\ell, z) \in \Vce \times \C^{n+1}: 0<|z| < 1+\eps\}$. The union $\Phi^{-1} ({\Sce}) \cup \left(\Vce \times \{0\}\right)$ will be denoted by $\Ace$.

\medskip

$\Ace$ is an open set in $\widetilde{\C}^{n+1}$ but we will endow it with other almost complex structures, different from $I_0$, so $\Ace$ should be thought of just as an oriented manifold and the structure on it will be specified in every instance. 

\medskip

We will keep using the same letters $\Phi^{-1}$ and $\P$ to denote the restricted maps 

\be
\label{eq:mapsrestr}
\begin{split}
\P^{-1}: \; {\Sc}  \rightarrow {\Ac} \;\;\;\; \\
\P: \; {\Ac} \rightarrow {\Sc} \cup \{0\} 
\end{split}
\ee

also when we look at these spaces just as oriented manifolds (not complex ones). We will make use of the notation

$$\Scr :=\Sc  \cap B_\rho^{2n+2}  \text{ and } \Acr := \Phi^{-1} ({\Scr}) \cup \left(\Vc \times \{0\}\right).$$

It should be kept in mind that $\Phi^{-1}$ and $\Phi$ in (\ref{eq:mapsrestr}) can be extended a bit beyond their boundaries, namely to $\Sce$ and to $\Ace:= \Phi^{-1} (\Sce) \cup \left(\Vce \times \{0\}\right)$.

\medskip

\begin{figure}[h]
\centering
 \includegraphics[width=9cm]{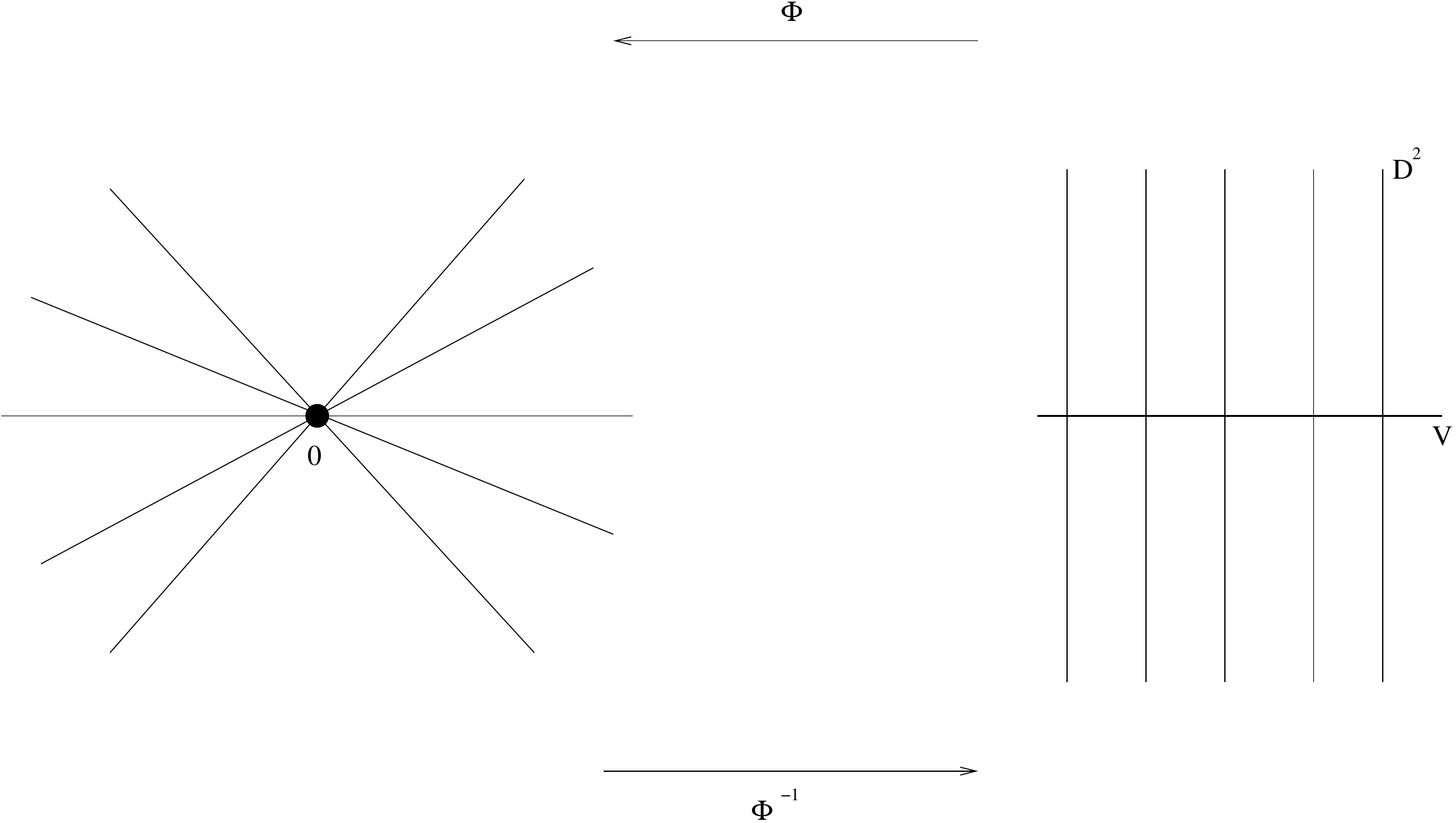}
\caption{Blowing up the origin. The maps $\P^{-1}: {\Sc}  \rightarrow {\Ac}$ and $\P: {\Ac} \rightarrow {\Sc} \cup \{0\} $.}
 \label{fig:blowup}
\end{figure}

Define on ${\Ac} \,\setminus (\CP^n \times \{0\})$:

\begin{itemize}
 \item the almost complex structure $I := {\P}^* J$, i.e. $I(\cdot):= {\pt}_* J {\P}_*(\cdot)$,
 
 \item the metric $\G(\cdot, \cdot):= \G_0(\cdot, \cdot) + \G_0(I \cdot, I \cdot)$,
 
 \item the non-degenerate two-form $\vt(\cdot, \cdot):= \G(I \cdot, \cdot) = \G_0(I \cdot, \cdot) - \G_0(\cdot, I \cdot)$.
\end{itemize}

The triple $(I, \G, \vt)$ is smooth on ${\Ac} \setminus (\CP^n \times \{0\})$ and makes it an almost complex manifold. We do not know yet, however, the behaviour of $(I, \G, \vt)$ as we approach $\Vc \times \{0\}$.

\begin{lem}[\textbf{the new structure is Lipschitz}]
\label{lem:lipcontrolI}
The almost complex structure $I$ fulfils 
$$|I-I_0|(\cdot) \leq c \text{dist}_{\G_0}(\; \cdot \; ,\;\CP^n \times \{0\}),$$
for $c= C \cdot Q$, where $C$ is a dimensional constant and $Q$ is as in the hypothesis on $J$ (paragraph ``new setting'', just before proposition \ref{Prop:restated}). $I$ can thus be extended continuously across across $\CP^n \times \{0\}$. 

Analogously we have $|\G-\G_0|(\cdot) \leq c \text{dist}_{\G_0}(\; \cdot \; ,\;\CP^n \times \{0\})$ and $|\vt-\vt_0|(\cdot) \leq c \text{dist}_{\G_0}(\; \cdot \; ,\;\CP^n \times \{0\})$. The triple $(I, \G, \vt)$ can be extended across $\CP^n \times \{0\}$ to the whole of ${\Ac}$ by setting it to be the standard $(I_0, \G_0, \vt_0)$ on $\CP^n \times \{0\}$. The structures $I, \G, \vt$ so defined are globally Lipschitz-continuos on $\Ac$, with Lipschitz constant $L + C \cdot Q$, where $L>0$ is an upper bound for the Lipschitz constants of $I_0$, $\G_0$ and $\vt_0$ (with respect to euclidean coordinates on $\Vc \times D^2$).
\end{lem}

\begin{proof}[\textbf{proof of lemma \ref{lem:lipcontrolI}}]
Recall that $\P$ is holomorphic for the standard structures $J_0$ and $I_0$. With respect to the flat metric on $\Sc$, we can choose an orthonormal basis at any point $q \neq 0$ made as follows: $$\{L_1, J_0(L_1), L_2, J_0(L_2), ... , L_n, J_0(L_n), W, J_0(W) \},$$ where $W$ and $J_0(W)$ span the $J_0$-complex $2$-plane through the origin and $q$. The map ${\pt}_*$ is holomorphic and sends this basis to one at $\pt(q) \in \Ac$, sending $W$ and $I_0(W)$ to a pair of vectors spanning the fiber through $\pt(q)$. On the vertical vectors ${\pt}_*$ is length preserving, while for the others $|{\pt}_* L_j| =|{\pt}_* J_0(L_j)|=\frac{\sqrt{1+|q|^2}}{|q|}$, as one can compute from the explicit expression of the Fubini-Study metric.

\medskip

Reversing this construction we can choose two basis, respectively at $p$ and $q=\P(p)$, as follows:

$$\{H_1, I_0(H_1), ... , H_n, I_0(H_n), V, I_0(V) \}$$ made of $\G_0$-unit vectors with scalar products w.r.t $\G_0$ bounded by $\frac{|q|}{\sqrt{1+|q|^2}}$, and

$$\left\{\frac{\sqrt{1+|q|^2}}{|q|} K_1, \frac{\sqrt{1+|q|^2}}{|q|} J_0(K_1),  ... , \frac{\sqrt{1+|q|^2}}{|q|} K_n, \frac{\sqrt{1+|q|^2}}{|q|} J_0(K_n), W, J_0(W) \right\} ,$$  

orthonormal at $q=\P(p)$, such that:

\begin{description}
 \item[(i)] $K_j:= {\P}_*H_j$ and $W:= {\P}_*V$;
 
 \item[(ii)] $V$ and $I_0(V)$ are vertical, i.e. they span the vertical fiber through $p$: by \textbf{(i)}, $W$ and $J_0(W)$ span the $J_0$-complex $2$-plane through the origin and $q$.
\end{description}

By the assumption that $J$ is close to $J_0$ in $B_1$ we can write the action of $J$ on $K_1$ as 

\be
\label{eq:actionH1}
\begin{split}
J(K_1) = (1+\la) J_0(K_1) + \sum_{j=1}^n \mu_j K_j + \sum_{j=2}^n \ti{\mu}_j J_0(K_j) + \\
+\frac{|q|}{\sqrt{1+|q|^2}}\sigma W_1 +\frac{|q|}{\sqrt{1+|q|^2}} \ti{\sigma}J_0(W_1).
\end{split}
\ee

Here $\la$, $\mu_j$, $\ti{\mu}_j$, $\sigma$ and $\ti{\sigma}$ are functions on $\Sc$ depending on $J-J_0$, evaluated at $q$, so their moduli are bounded by $|J-J_0|(q)< Q |q|$. 

\medskip

Let us write the action of $I$ on $H_1$ explicitly: by definition of $I$, using (\ref{eq:actionH1}),

$$I(H_1):= {\pt}_* J {\P}_*(H_1) = {\pt}_* J (K_1) =$$

\be
\label{eq:horizlipest}
\begin{split}
=((1+\la) \circ \P) I_0(H_1) + \sum_{j=1}^n (\mu_j \circ \P) H_j + \sum_{j=2}^n (\ti{\mu}_j \circ \P) I_0(K_j) + \\
+\frac{|q|}{\sqrt{1+|q|^2}}(\sigma \circ \P) V_1 +\frac{|q|}{\sqrt{1+|q|^2}} (\ti{\sigma} \circ \P) I_0(V_1).
\end{split}
\ee

\medskip

Similar expressions are obtained for the actions on $H_j$ and $I_0(H_j)$ for all $j$. Now

$$J(W) = \sigma W + (1+\ti{\sigma})J_0(W),$$

since the $2$-plane spanned by $W$ and $J_0(W)$ is $J$-pseudo holomorphic by hypothesis.

Here $\sigma$ and $\ti{\sigma}$ are functions on $\Sc$ depending on $J-J_0$, evaluated at $q$, and their moduli are bounded by $|J-J_0|< Q |q|$. 

So the action of $I$ on $V$ is explicitly given by

$$I(V):= {\pt}_* J {\P}_*(V) = {\pt}_* J (W) =$$ $$=(\sigma \circ \P) {\pt}_* (W) + ((1+\ti{\sigma}) \circ \P) {\pt}_* J_0(W)$$ 

\be
\label{eq:vertlipest}
=(\sigma \circ \P) V + ((1+\ti{\sigma}) \circ \P) I_0(V).
\ee

So we have, from (\ref{eq:horizlipest}) and (\ref{eq:vertlipest}) that there exists $c = C \cdot Q$ (for some dimensional constant $C$) such that $(I-I_0)$ at the point $p=\pt(q)$ has norm $\leq c |q| = c \;\text{dist}_{\G_0}(\; \cdot \; ,\;\CP^n \times \{0\})$.

The analogous estimates on $\G$ and $\vt$ follow by their definition. So we can extend the triple $(I, \G, \vt)$ across $\CP^n \times \{0\}$ in a Lipschitz continuous fashion.

From (\ref{eq:horizlipest}) and (\ref{eq:vertlipest}) we also get that $I$ is, globally in ${\Ac}$, a Lipschitz continuous perturbation of $I_0$, and the same goes for $\G$ and $\vt$: indeed the Lipschitz constants of $\la$, $\mu_j$, $\ti{\mu}_j$, $\sigma$ and $\ti{\sigma}$ are controlled by $C \cdot Q$, for some dimensional constant $C$ (which can be taken the same as the $C$ we had above, by choosing the larger of the two). 

\end{proof}

\begin{oss}
The importance of working with coordinates adapted to $J$, as chosen in section \ref{polar}, relies in the fact that this allows to obtain the Lip\-schitz extension across $\CP^n \times \{0\}$, which could fail on the vertical vectors if coordinates were taken arbitrary. 
\end{oss}

\medskip

\medskip

\medskip

The aim is now to \textbf{translate our problem} in the new space $({\Ac}, I, g, \vt)$. The trouble is that the push-forward of $T$ via $\P^{-1}$ can only be done away from the origin and the map $\P^{-1}$ degenerates as we get closer to $0$. 

\medskip

For any $\rho>0$ we can take the \textit{proper transform} of $T \res (\Sc \setminus \Scr )$ by pushing forward via $\P^{-1}$, since this is a diffeomorphism away from the origin:

$$P_{\rho}:=  {\pt}_* \left(T \res (\Sc \setminus \Scr )\right).$$

What happens when $\rho \to 0$ ? The following two lemmas yield the answers.

\begin{lem}
\label{lem:finitemasscycle}
The current $P:=\lim_{\rho \to 0}  P_\rho = \lim_{\rho \to 0} {\pt}_* \left(T \res (\Sc \setminus \Scr)\right)$ is well-defined as the limit of currents of equibounded mass to be a current of finite mass in $\Ac$. 

The mass of $P$, both with respect to $\G$ and to $\G_0$, is bounded by a dimensional constant $C$ times the mass of $T$.
\end{lem}

\begin{lem}
\label{lem:normalposcycle}
The current $P:=\lim_{\rho \to 0}  P_\rho = \lim_{\rho \to 0} {\pt}_* \left(T \res (\Sc \setminus \Scr)\right)$ is a $\vt$-positive normal cycle in the open set ${\Ac}$ ($\vt$ is a semi-calibration with respect to $\G$).
\end{lem}

\medskip

A little notation before the proofs. For any $\rho$ consider the dilation $\lambda_\rho(\cdot):=\frac{\cdot}{\rho}$, sending $B_\rho$ to $B_1$, and the map 

\be
\label{eq:dilblow}
\Lambda_\rho:\Acr \to \Ac, \;\; \Lambda_\rho:=\Phi^{-1} \circ \lambda_\rho \circ \Phi ,
\ee

which in the coordinates of $\CP^n \times \C^{n+1}$ (the ambient space in which $\Ac$ is embedded) reads $\Lambda_\rho (\ell, z) = \left(\ell, \frac{z}{\rho}\right)$.

\begin{proof}[\textbf{proof of lemma \ref{lem:finitemasscycle}}]
The currents $T$ and $T_{0,r}:=(\la_r)_* (T \res B_\rho)$ are defined in $\Sce$ and by remark \ref{oss:mononewsett}, i.e. by the monotonicity formula, we have a uniform bound on the masses: $M(T_{0,r}) \leq K$.

\medskip

The map $\P^{-1}$ is pseudo holomorphic with respect to $J$ and $I$ by definition of $I$; thus each $P_\rho = {\pt}_* \left(T \res (\Sc \setminus \Scr)\right)$ is $\vt$-positive by construction (see lemma \ref{lem:A1}), so $M(P_\rho)=P_\rho(\vt)$, where the mass is computed here with respect to $\G$, the metric defined before lemma \ref{lem:lipcontrolI}. The currents $P_{\rho}$ and $P_{\rho'}$, for $\rho > \rho'$, coincide on $\Ac \setminus \overline{\Acr}$, therefore in order to study the limit as $\rho \to 0$, it is enough to look at a chosen sequence $\rho_k \to 0$ and prove that $P_{\rho_k}$ have equibounded masses and thus converge to a limit $P$, which must then be the limit of the whole family $P_\rho$.

\medskip

\textit{1st step: choice of the sequence}. Denote by $\langle T, |z|=r\rangle$ the slice of a current $T$ with the sphere $\p B_r$. Choose $\rho_k$ so to ensure

\begin{itemize}
 \item (i) $T_{\rho_k} \rightharpoonup T_\infty$ in ${\Sc}$ for a certain cone $T_\infty$,
 
 \item (ii) $M(\langle T_{\rho_k}, |z|=1 \rangle)$ are equibounded by $4K$.
\end{itemize}

This is achieved as follows: take a sequence $\rho_k'$ fulfilling (i); remark \ref{oss:mononewsett} tells us that $M(T_{\rho_k'})$ are equibounded by a constant $K$ independent of $k$. By slicing theory (see \cite{G})

$$\int_{\frac{1}{2}}^1 M(\langle T_{\rho_k'}, |z|=r \rangle) dr \leq M(T_{\rho_k'} \res (B_1 \setminus B_{\frac{1}{2}}))\leq K,$$

thus at least half of the slices $\langle T_{\rho_k'}, |z|=r \rangle_{r \in [\frac{1}{2},1]}$ have masses $\leq 2K$. For every $k$ we can choose $\frac{1}{2} \leq s_k \leq 1$ such that all the slices $\langle T_{\rho_k'}, |z|=s_k \rangle$ exist and have mass $\leq 2K$. Then with $\rho_k = s_k \rho_k'$ it holds 

$$M(\langle T_{\rho_k}, |z|=1 \rangle) = M\left( \left({\la}_{\frac{\rho_k}{\rho_k'}}\right)_* \left\langle {T_{\rho_k'}}, |z|=s_k={\frac{\rho_k}{\rho_k'}} \right\rangle \right) \leq 2 \cdot 2K$$

and since $\frac{\rho_k'}{2} \leq \rho_k \leq \rho_k'$ the sequence  $T_{\rho_k}$ also converges to the same $T_\infty$.

\medskip

Since $\langle T_{\rho_k}, |z|=1 \rangle = \left(\la_{\rho_k}\right)_* \langle T, |z|=\rho_k \rangle $, condition (ii) also reads

\be
\label{eq:stimaslice}
M \left( \langle T, |z|=\rho_k \rangle \right) \leq 4K \rho_k.
\ee

\medskip

\textit{2nd step: uniform bound on the masses}. We use in $\Ac$ standard coordinates inherited from $\CP^n \times \C^{n+1}$, i.e. we have $2n$ horizontal variables (from $\CP^n$) and $2n+2$ vertical variables. 

The standard symplectic form $\vt_0$ is $\mathcal{E}^*(\vt_{\CP^n} + \vt_{\C^{n+1}})$, as in the beginning of section \ref{blowup}. We want to estimate $M(P_\rho)=P_\rho(\vt) = P_\rho(\vt_0)+P_\rho(\vt-\vt_0)$. 

\medskip 

Let us first deal with $P_\rho(\mathcal{E}^*\vt_{\CP^n}) = T_{0,\rho} \res (\Sc \setminus \Scr)({\pt}^*\mathcal{E}^*\vt_{\CP^n})$. It is convenient here to keep in mind that $\vt_0$ is actually defined on $\Ace$ and consider ${\pt}^*\mathcal{E}^*\vt_{\CP^n}$ as a form on $\Sce$, since $\P^{-1}$ also extends to $\Sce$. The map $\mathcal{E} \circ \P^{-1}: \Sce \to \Ace$ has the coordinate expression $(z_0, ... z_n) \to \left( (\frac{z_1}{z_0}, ..., \frac{z_n}{z_0}), (z_0, ... z_n)\right) \in \Vce \times \C^{n+1}$, using the chart $z_0 \neq 0$ on $\Vce \subset \CP^n$.

Using the explicit expression of $\vt_{\CP^n}$ (see \cite{MS} or the beginning of this section) we can write in the domain $\Sce$, where $z_0 \neq 0$, 

$${\pt}^*\mathcal{E}^*(\vt_{\CP^n})=\p \overline{\p} \log \left( 1 + \sum_{j=1}^n \frac{|z_j|^2}{|z_0|^2} \right).$$ 

We are neglecting a factor $\frac{i}{2}$, which would not play any significant role in this proof. In particular ${\pt}^*\mathcal{E}^*(\vt_{\CP^n}) = d \eta$, where 

$$\eta=\frac{1}{2}\left( \overline{\p} \log \left( 1 + \sum_{j=1}^n \frac{|z_j|^2}{|z_0|^2} \right) - \p \log \left( 1 + \sum_{j=1}^n \frac{|z_j|^2}{|z_0|^2} \right)\right).$$

We thus have 

$$P_\rho(\mathcal{E}^* \vt_{\CP^n}) =  \left(T \res (\Sc \setminus \Scr)\right)({\pt}^* \mathcal{E}^* \vt_{\CP^n}) = \left(T \res (\Sc \setminus \Scr)\right)( d \eta) =$$

$$=\p\left[T \res (\Sc \setminus \Scr)\right] \left(\eta \right).$$

The boundary of $T \res (\Sc \setminus \Scr)$ is made of three portions: two live in the spheres $\p B_1$ and $\p B_\rho$ and the third one is given by the slice of $(T \res \Sce) \res (B_1 \setminus B_\rho)$ with the hypersurface $\sum_{j=1}^n \frac{|z_j|^2}{|z_0|^2} = 1$. There is no loss of generality in assuming that these slices exists. 

The explicit form of $\eta$ then implies that the latter portion of boundary, i.e. the slice of $T$ with the hypersurface $\sum_{j=1}^n \frac{|z_j|^2}{|z_0|^2} = 1$,  has zero action on $ \eta$. We can thus write

$$P_\rho(\mathcal{E}^* \vt_{\CP^n}) =  \langle T \res \Sc, |z|=1 \rangle \left(\eta \right) -  \langle T \res \Sc, |z|=\rho \rangle   \left( \eta \right).$$

Now observe the comass of $ \eta $. The comasses are equivalent up to a universal constant $C$ to the maximum modulus of the coefficients of the form. We can explicitly compute $\|\eta\|^* \leq \frac{C}{\rho}$, where $\rho$ is the distance from the origin.

Now we focus on the sequence $\rho_k$ chosen in step 1, for which (ii) and (\ref{eq:stimaslice}) hold. We thus get, independently of $\rho_k$,

\be
\label{eq:firstbound}
|P_{\rho_k}(\mathcal{E}^* \vt_{\CP^n})| \leq 4 K\, C .
\ee

The estimate 

\be 
\label{eq:secondbound}
|P_{\rho}(\mathcal{E}^*\vt_{\C^{n+1}})| = |T_{0,\rho} \res (\Sc \setminus \Scr)({\pt}^*\mathcal{E}^*\vt_{\C^{n+1}})|\leq  K
\ee

follows easily since $\P^{-1}$ is lenght-preserving in the vertical coordinates and thus $(\mathcal{E} \circ \P^{-1})^*$ preserves the comass of $\vt_{\C^{n+1}}$.

\medskip

Now let us consider $|P_\rho(\vt-\vt_0)|$. Thanks to the Lipschitz control from lemma \ref{lem:lipcontrolI}, i.e. $|\vt-\vt_0|(\cdot) \leq   c \text{dist}_{\G_0}(\cdot,\CP^n \times \{0\})$, the two-form ${\pt}^* (\vt-\vt_0)$ in $\Sc$ has comass $\leq \frac{c \cdot C}{\rho} \leq \frac{C}{\rho}$, where $\rho$ is the distance from the origin and $C$ is a dimensional constant ($c$ can be assumed to be smaller that $1$). 

We can then decompose $\Sc = \cup_{j=0}^\infty A_j$, where $A_j = \Sc \cap \left(B_{\frac{1}{2^j}} \setminus B_{\frac{1}{2^{j+1}}}\right)$. As observed in remark \ref{oss:mononewsett} it holds $M(T \res A_j) \leq K \frac{1}{2^{2j}}$. On the other hand the comass of ${\pt}^* (\vt-\vt_0)$ in $A_j$ is $\leq  C \, 2^{j+1}$.

Therefore summing on all $j$'s we can bound 

\be
\label{eq:thirdbound}
|P_\rho(\vt-\vt_0)|= \left|(T \res \Sc)\left({\pt}^* (\vt-\vt_0)\right)\right|  \leq $$ $$\leq K \, C \,\sum_{j=0}^\infty  2^{j+1} \frac{1}{2^{2j}} = K \, C \, \sum_{j=0}^\infty 2^{1-j} = 4K\,C,
\ee

so $|P_\rho(\vt-\vt_0)|$ is also equibounded independently of $\rho$.

\medskip

Putting (\ref{eq:firstbound}), (\ref{eq:secondbound}) and (\ref{eq:thirdbound}) together, we obtain that $M(P_{\rho_k})$ are uniformly bounded by $K$ times a dimensional constant $C$. By compactness there exists a current $P$ in ${\Ac}$ such that $P_{\rho} \rightharpoonup P$. 

So far we were taking the mass with respect to $\G$. Since $\G$ is $c$-close to $\G_0$, for a small constant $c$, an analogous bound holds, up to doubling the constant $C$, for the mass of $P$ computed with respect to $\G_0$. This observation is needed later in section \ref{proof}. 

\end{proof}

\medskip

Our next aim is to prove that the current $P$ just obtained is in fact a cycle in the open set $\Ac$. A priori this is not clear, for in the limit $\rho \to 0$ some boundary could be created on $\CP^n \times \{0\}$.

\medskip

\begin{proof}[\textbf{proof of lemma \ref{lem:normalposcycle}}]

\textit{Step 1}. We are viewing $P$ as a current in the open set $\Ac$ in the manifold $\widetilde{\C}^{n+1}$, so the same should be done for the currents $P_{\rho}:= {\pt}_*\left(T \res (\Sc \setminus \Scr)\right)$. Given a sequence $\rho_k \to 0$, we want to observe the boundaries $\p P_{\rho_k}$. Up to a subsequence we may assume that $\rho_k$ is such that $T_{0, \rho_k} \rightharpoonup T_\infty$ for a certain cone. Then the boundaries $\p P_{\rho_k}$ satisfy, as  $k \to \infty$, by the definition (\ref{eq:dilblow}) of $\La_{\rho_k}$:

\be
\label{eq:boundariesdilblow}
(\Lambda_{\rho_k})_*(\p P_{\rho_k}) =  -{\pt}_* \langle T_{0, \rho_k}, |z|=1 \rangle \rightharpoonup -{\pt}_* \langle T_\infty, |z|=1 \rangle .
\ee

Recall that we are viewing $P_{\rho_k}$ as currents in the open set $\Ac$, so also $T \res (\Sc \setminus \Scr)$ should be thought of as a current in the open set $\Sc$: this is why the only boundary comes from the slice of $T$ with $|z|=\rho_k$.

\medskip

Moreover if the sequence is chosen (and we will do so) as in step 1 of lemma \ref{lem:finitemasscycle}, then $(\Lambda_{\rho_k})_*(\p P_{\rho_k})$ have equibounded masses, since so do $\p (T_{0, \rho_k})$ and $\P^{-1}$ is a diffeomorphism on $\p B_1$.

The current $T_\infty$ has a special form: it is a $(1,1)$-cone, so the $1$-current $\langle T_\infty, |z|=1 \rangle$ has an associated vector field that is always tangent to the Hopf fibers\footnote{The Hopf fibration is defined by the projection $H: S^{2n+1} \subset \C^{n+1} \rightarrow \CP^n$, $H(z_0, ..., z_n) = [z_0, ..., z_n]$. The Hopf fibers $H^{-1}(p)$ for $p \in \CP^n$ are maximal circles in $S^{2n+1}$, namely the links of complex lines of $\C^{n+1}$ with the sphere.} of $S^{2n+1}$. 

\medskip

\textit{Step 2}. We want to show that $P$ is a cycle in $\Ac$, i.e. that $\p P_{\rho_k} \to 0$ as $n \to \infty$. The boundary in the limit could possibly appear on $\CP^n \times \{0\}$ and we can exclude that as follows.

Let $\alpha$ be a $1$-form of comass one with compact support in $\Ac$ and let us prove that $\p P_{\rho_k}(\alpha) \to 0$. Since $\Ac$ is a submanifold in $\CP^n \times \C^{n+1}$, we can extend $\alpha$ to be a form in $\CP^n \times \C^{n+1}$. Let us write, using horizontal coordinates $\{t_j\}_{j=1}^{2n}$ on $\CP^n$ and vertical ones $\{s_j\}_{j=1}^{2n+2}$ for $\C^{n+1}$, $\alpha=\alpha_h + \alpha_v$, where $\alpha_h$ is a form in the $d t_j$'s, $\alpha_v$ in the $d s_j$'s. Rewrite, viewing $P_{\rho_n}$ as currents in $\CP^n \times \C^{n+1}$,

$$\p P_{\rho_k}(\alpha) = \left[(\Lambda_{\rho_k})_*(\p P_{\rho_k})\right]\left( \Lambda_{\rho_k}^{-1})^* \alpha \right).$$

The map $\Lambda_{\rho_k}^{-1}$ is expressed in our coordinates by $(t_1, ..., t_{2n}, s_1, ...s_n) \to (t_1, ..., t_{2n}, \rho_k s_1, ... \rho_k s_{2n+2})$, therefore 

$$(\Lambda_{\rho_k}^{-1})^* \alpha = \alpha^n_h + \alpha^n_v,$$

where the decomposition is as above and with $\|\alpha^n_h\|^* \approx \|\alpha_h\|^*$ and $\|\alpha^n_v\|^* \lesssim \rho_k \|\alpha_v\|^*$. The signs $\approx$ and $\lesssim$ mean respectively equality and inequality of the comasses up to a dimensional constant, so independently of the index $n$ of the sequence. 

\medskip

As $k \to \infty$ it holds $\alpha^k_h \to \alpha^\infty_h$ in some $C^\ell$-norm, where $\|\alpha^\infty_h\|^* \lesssim 1$ and $\alpha^\infty_h$ is a form in the $d t_j$'s. More precisely $\alpha^\infty_h$ coincides with the restriction of $\alpha_h$ to $\CP^n \times \{0\}$, extended to $\CP^n \times \C^{n+1}$ independently of the $s_j$ variables. We can write

$$\left|\left[(\Lambda_{\rho_k})_*(\p P_{\rho_k})\right]( \alpha^k_h)\right| \leq \left|\left[(\Lambda_{\rho_k})_*(\p P_{\rho_k})\right]( \alpha^k_h - \alpha^\infty_h)\right| + \left|\left[(\Lambda_{\rho_k})_*(\p P_{\rho_k})\right]( \alpha^\infty_h)\right| $$

and both terms on the r.h.s. go to $0$. The first, since $M((\Lambda_{\rho_k})_*(\p P_{\rho_k}))$ are equibounded and $|\alpha^k_h - \alpha^\infty_h|  \to 0$; the second because we can use (\ref{eq:boundariesdilblow}) and ${\pt}_* \p (T_\infty)$ has zero action on a form that only has the $d t_j$'s components, as remarked in step 1.

\medskip

Moreover 

$$\left|\left[(\Lambda_{\rho_k})_*(\p P_{\rho_k})\right]( \alpha^k_v)\right|  \to 0, $$

because $(\Lambda_{\rho_k})_*(\p P_{\rho_k}) =  -{\pt}_* \langle T_{0, \rho_k}, |z|=1 \rangle$ have equibounded masses by the choice of $\rho_k$, while $\|\alpha^k_v\|^* \lesssim \rho_k \|\alpha_v\|^*$ have comasses going to $0$.

\medskip

Therefore no boundary appears in the limit and $P$ is a normal cycle in $\Ac$. The fact that it is $\vt$-positive follows easily by the fact that so are the currents $P_\rho$, as remarked in the beginning of the proof of lemma \ref{lem:finitemasscycle}.

\end{proof}

\medskip

Summarizing, we define the current $P$ just constructed to be the \textbf{proper transform} of the positive-$(1,1)$ normal cycle $T \res \Sc$. $P$ is a normal and $\vt$-positive cycle in $\Ac$, where the semicalibration $\vt$ is Lipschitz (and actually smooth away from $\CP^n \times \{0\}$). Therefore the almost monotonicity formula holds true for $P$. Observe that the metric $\G$ on $\Ac$ fulfils the hypothesis $\frac{1}{5} \mathbb{I} \leq \G \leq 5 \mathbb{I}$ of proposition \ref{Prop:monotonicity}, being a perturbation of $\G_0$, which is in turn built from the Fubini-Study metric.

\section{Proof of the result}
\label{proof}

With the assumptions in proposition \ref{Prop:restated}, we have to observe the family $T_{0,r}=(\la_{r})_* T$ as $r \to 0$. These currents have equibounded masses by (\ref{eq:newmonotone}).

\medskip

Take any converging sequence $T_{0,r_n}:=(\la_{r_n})_* T \rightharpoonup T_\infty$ for $r_n \to 0$. Take the proper transform of each $T_{0,r_n}$ and denote it by $P_n$. Lemmas \ref{lem:finitemasscycle} and \ref{lem:normalposcycle} yield that $P_n$ is a $\vt_n$-positive cycle, for a semicalibration $\vt_n$ in the manifold $\Ac$.
$\vt_n$ is smooth away from $\Vc \times \{0\}$ and it is Lipschitz-continuous, with $|\vt_n-\vt_0| < c_n \text{dist}_{\G_0}(\cdot, \CP^n \times \{0\})$. Recalling lemma \ref{lem:lipcontrolI} we can see that, since the almost complex structure $J_{r_n}$ on $\Sc$ fulfils $|J_{r_n}-J_0| < (Q \, r_n) \cdot r$ in $\Sc$ (by dilation), then the constants $c_n$ go to $0$ as $n \to \infty$. Analogously we get that the Lipschitz constants of $\vt_n$ are uniformly bounded by $2L$. 

By lemma \ref{lem:finitemasscycle} the masses of $P_n$ are uniformly bounded in $n$ (with respect to $\G_0$), since so are the masses of $T_{0,r_n}$, $M(T_{0,r_n})\leq K$. 

So by compactness, up to a subsequence that we do not relabel, we can assume $P_n \rightharpoonup P_\infty$ as $n \to \infty$ for a normal cycle $P_\infty$. 

\medskip

\begin{lem}
 \label{lem:pinfty}
 $P_\infty$ is a $\vt_0$-positive cycle; more precisely it is the proper transform of $T_\infty$.
\end{lem}

\begin{proof}
$\vt_0$-positiveness follows straight from the $\vt_n$-positiveness of $P_n$ and $|\vt_n-\vt_0| < c_n \text{dist}_{\G_0}(\cdot, \CP^n \times \{0\})$, $c_n \to 0$.

Recall that $\vt_0=\mathcal{E}^*(\vt_{\CP^n} + \vt_{\C^{n+1}})$; we want to estimate (notation from section \ref{blowup})
$$M(P_\infty \res \Acr)= (P_\infty \res \Acr)(\vt_0) = \lim_{n \to \infty} (P_n \res \Acr)(\vt_0).$$ 

Write

\be
\label{eq:est1}
(P_n \res \Acr)(\vt_0)=(P_n \res \Acr)(\mathcal{E}^*\vt_{\CP^n}) + (P_n \res \Acr)(\mathcal{E}^*\vt_{\C^{n+1}}).
\ee

Let us bound the second term on the r.h.s.

$$(P_n \res \Acr)(\mathcal{E}^*\vt_{\C^{n+1}}) = (\La_\rho)_*(P_n \res \Acr) \left((\La_r^{-1})^*(\mathcal{E}^*\vt_{\C^{n+1}})\right).$$

The current $(\La_\rho)_*(P_n \res \Acr)$ is the proper transform of $T_{0, \rho r_n}$, therefore (lemma \ref{lem:finitemasscycle}) $M\left((\La_\rho)_*(P_n \res \Acr)\right) \leq K \, C$ independently of $n$; the form $(\La_r^{-1})^*(\mathcal{E}^*\vt_{\C^{n+1}})$ has comass bounded by $\rho^2$. Altogether 

$$(P_n \res \Acr)(\mathcal{E}^*\vt_{\C^{n+1}}) \leq  K \, C \, \rho^2 .$$

To bound the first term on the r.h.s. of (\ref{eq:est1}), let $P$ be the proper transform of $T$; using that $(\La_{r_n})^* \mathcal{E}^*\vt_{\CP^n} = \mathcal{E}^*\vt_{\CP^n}$ we can write

$$(P_n \res \Acr)(\mathcal{E}^*\vt_{\CP^n}) = (P \res {\Ac}_{r_n \rho}) (\mathcal{E}^*\vt_{\CP^n}) \leq M\left( P \res {\Ac}_{r_n \rho} \right) \leq M\left( P \res {\Ac}_{\rho} \right), $$

which goes to $0$ as $\rho \to 0$ by lemma \ref{lem:finitemasscycle}. Summarizing we get that there exists a function $o_\rho(1)$ that is infinitesimal as $\rho \to 0$, such that $|(P_n \res \Acr)(\vt_0)| \leq o_\rho(1)$ (the point is that $o_\rho(1)$ can be chosen independently of $n$). 

Therefore also $M(P_\infty \res \Acr) = \lim_{n \to \infty} (P_n \res \Acr)(\vt_0)\leq o_\rho(1)$, which means that  

\be
\label{eq:limitPinfty}
P_\infty = \lim_{\rho \to 0} P_\infty \res (\Ac \setminus \Acr).
\ee

\medskip

Recall now that the proper transform is a diffeomorphism away from the origin, thus

$$P_\infty \res (\Ac \setminus \Acr) = \lim_n {\pt}_* T_{0,r_n} \res (\Sc \setminus \Scr) = {\pt}_* T_\infty \res (\Sc \setminus \Scr),$$

which concludes, together with (\ref{eq:limitPinfty}), the proof that $P_\infty$ is the proper transform of ${\pt}_* T_\infty$.

\end{proof}

Recalling (\ref{eq:11link}), the previous lemma tells us that $P_\infty$ is of a very special form. Denoting $\Vc:=\left\{\sum_{j=1}^n \frac{|z_j|^2}{|z_0|^2} <1\right\} \subset \CP^n$ and, for each disk $D^X$ in $\Sc$, $L^X$ the disk such that $\P(L^X) = D^X$, we have

\be
\label{eq:11linktransf}
P_{\infty}(\beta) = \int_{\Vc} \left\{ \int_{L^X} \langle \beta, \vec{L^X} \rangle \; d \mathcal{L} ^{2} \right\} d \tau|_{\Vc}(X) .
\ee

When we take the proper transform the density is preserved going from $\Sc$ to $\P^{-1}(\Sc)$, since $\P^{-1}$ is a diffeomorphism on $\Sc$ (see lemma \ref{lem:A2}).

We are ready to conlcude the proof of proposition \ref{Prop:restated}, and therefore of theorems \ref{thm:main1} and \ref{thm:main2}. 

\begin{proof}[\textbf{proof of proposition \ref{Prop:restated}}]
The points $\frac{x_m}{|x_m|}$ converge to the point $(1, 0, ..., 0)$ in $D \cap S^{2n+1}$, where $D$ is the disk $D=D^{[1,0,...0]}$. 

We want to show that any converging sequence $T_{0,r_n}:=(\la_{r_n})_* T \rightharpoonup T_\infty$ is such that the cone $T_\infty$ contains $\ka \llbracket D \rrbracket$.

\medskip

Let us apply the proper transform to $T_{0,r_n}$ and get $P_n$ as in lemma \ref{lem:pinfty}. Fix $n$: there is a sequence $\{x_m\}$ tending to the origin of points with densities such that $\liminf_{m \to \infty} \nu(x_m) \geq \ka$. By lemma \ref{lem:A2} the points $p_m := {\pt}(x_m)$ also have densities fulfilling that their $\liminf$ is $\geq \ka$ for $P_n$. 

It easily seen that it holds $p_m \to p_0=([1,0,...0],0) \in \CP^n \times \C^{n+1}$.

By upper semi-continuity of the density (which follows from the almost monotonicity formula for $P_n$) we get that $p_0$ also has density $\geq \ka$ for $P_n$. 

Doing this for every $n$ we get that we are dealing with a sequence of normal cycles $P_n$ all having the point $p_0$ as a point of density $\geq \ka$. We wish to prove that, being the cycles $P_n$ positive, then the point $p_0$ is also of density $\geq \ka$ for the limit $P_\infty$.

The cycles $P_n$ are $\vt_n$-positive so for any $\delta>0$ it holds

$$M(P_n \res B_\delta(p_0)) =  (P_n \res B_\delta(p_0))(\vt_n).$$

By weak convergence

$$M(P_\infty \res B_\delta(p_0)) =  (P_\infty \res B_\delta(p_0))(\vt_0) = $$ $$= \lim_{n \to \infty}(P_n \res B_\delta(p_0))(\vt_0). $$ 

We can split 

\be
\label{eq:splitvalue}
(P_n \res B_\delta(p_0))(\vt_0)=(P_n \res B_\delta(p_0))(\vt_0-\vt_n) + (P_n \res B_\delta(p_0))(\vt_n).
\ee

The semi-calibrations $\vt_n$ have uniform bounds on their Lipschitz constants, say $2L$. The metrics at $p_0$ coincide with $\G_0$ independently of $n$. We can therefore use the almost monotonicity formula for $P_n$ at $p_0$ (proposition \ref{Prop:monotonicity}) to get

$$(P_n \res B_\delta(p_0))(\vt_n)=M(P_n \res B_\delta(p_0)) \geq \pi (\ka - C 2L \delta) \delta^2 , $$

where $C$ is a universal constant. The forms $\vt_n$ fulfil $|\vt_n - \vt_0| < c_n$ in $B_\delta(p_0)$ and $c_n \to 0$ as $n \to \infty$.  Therefore we can bound, from (\ref{eq:splitvalue}),

$$|(P_n \res B_\delta(p_0))(\vt_0)| \geq -c_n K \, C + M(P_n \res B_\delta(p_0)) \geq -c_n K \, C+\pi \ka \delta^2- 2C L \delta^3 .$$

Since $c_n \to 0$ we can conclude 

\be
\label{eq:densitypassestolimit}
M(P_\infty \res B_\delta(p_0)) \geq \pi \ka \delta^2-2C L \delta^3
\ee

independently of $\delta$, which means that $p_0$ is a point of density $\geq \ka$ for the $\vt_0$-positive cycle $P_\infty$.

Recall the structure of $P_\infty$: it is made by the holomorphic disks $L^X$ weighted with the positive measure $\tau$, so if $y_0$ has density $\geq \ka$, then the disk $L^{[1,0, ...0]}$ must be weighted with a mass $\geq \ka$, in other words the measure $\tau$ must have an atom of mass $\geq \ka$ at $y_0$. 

So $P_\infty$ is of the form $\ka \llbracket L^{[1,0, ...0]} \rrbracket + \tilde{P}$, for a $\vt_0$-positive current $\tilde{P}$. Transforming back via $\Phi$, $T_\infty$ contains the disk $\ka \llbracket D \rrbracket$, as required.
 
\end{proof}

\newpage

\appendix
\section{Appendix}

The following almost-monotonicity formula for positive or semi-calibrated cycles is proved in \cite{PR}, Proposition 1, for a $C^1$ semi-calibration: the same proof works as well for a form with Lipschitz-continous coefficients, so we only give the statement.

\medskip

Let the ball of radius $2$ in $\R^d$ be endowed with a metric $g$ and a two-form $\om$ such that both $g$ and $\om$ have Lipschitz-continuous coefficients (with respect to the standard coordinates on $\R^n$) and $\om$ has unit comass for $g$. The metric $g$ is represented by a matrix and we further assume that $\frac{1}{5} \mathbb{I}\leq g \leq 5 \mathbb{I}$, where $\mathbb{I}$ is the identity matrix. So $g$ is a Lipshitz perturbation of the flat metric.

Let $T$ be a $\om$-positive normal cycle. Then we have a $2$-vector field $\vec{T}(x)$, of unit mass with respect to $g$. This means that for $\|T\|$-a.a. $x$, $\vec{T}(x)=\sum_{k=1}^{N(x)} \la_k(x) \vec{T}_k(x)$, a convex combination of $\om_x$-calibrated unit simple $2$-vectors. The mass refers pointwise to the metric $g_{x}$.

\begin{Prop}
\label{Prop:monotonicityapp}
In the previous hypothesis, there exists $r_0 >0$ and $C>0$, depending only on the Lipschitz constants of $g$ and $\om$ such that, given an arbitrary point $x_0 \in B_1(0)$, the following holds. 

Denote by $B_r(x_0)$ (respectively $B_s(x_0)$) the ball around $x_0$ of radius $r$ (respectively $s$) with respect to the metric $g_{x_0}$; let $|\cdot|$ be the distance for $g_{x_0}$ and $|\cdot|_g$ the mass-norm with respect to $g$. Let $\frac{\p}{\p r}$ be the unit radial vector field with respect to $x_0$ and $g_{x_0}$.

For any $0<s<r<r_0$, we have

\be
\label{eq:monotonicity}
\begin{split}
\frac{e^{Cr} + Cr}{r^{2}}  \left( T \res B_r(x_0) \right) (\om) - \frac{e^{Cs} + Cs}{s^{2}}  \left( T \res B_s(x_0) \right)(\om) \\
\geq  \int_{B_r \setminus B_s (x_0)} \frac{1}{|x-x_0|^{2}} \sum_{k=1}^{N(x)} \la_k(x)  \left| \vec{T}_k(x) \wedge \frac{\p}{\p r} \right|_{g(x)}^2  d \|T\|
\end{split}
\ee

and 

\be
\label{eq:monotonicity2}
\begin{split}
\frac{e^{Cr} - Cr}{r^{2}}  \left( T \res B_r(x_0) \right) (\om) - \frac{e^{Cs} - Cs}{s^{2}}  \left( T \res B_s(x_0) \right)(\om) \\
\leq  \int_{B_r \setminus B_s (x_0)} \frac{1}{|x-x_0|^{2}} \sum_{k=1}^{N(x)} \la_k(x)  \left| \vec{T}_k(x) \wedge \frac{\p}{\p r} \right|_{g(x)}^2  d \|T\| .
\end{split}
\ee

\end{Prop}

\newpage

The following two lemmas are used in the paper when pushing forward a positive cycle under a diffeomorphism.

\medskip

\begin{lem} \label{lem:A1} [\textbf{the pushforward of a positive-$(1,1)$ current via a pseudoholomorphic diffeomorphism is positive-$(1,1)$}]

Let $C$ be a positive-$(1,1)$ normal current in an open set $U \subset \R^{2N}$, endowed with an almost complex structure $J_1$, a compatible metric $g_1$ and a two-form $\om_1$. Let $f:U \rightarrow \R^{2N}$ be a smooth pseudoholomorphic diffeomorphism, where $\R^{2N}$ is endowed with an almost complex structure $J_2$ and compatible metric and semi-calibration $g_2$ and $\om_2$. Then $f_* C$ is a positive-$(1,1)$ normal current in $(\R^{2N}, J_2, g_2)$.
\end{lem}

\begin{proof}[\bf{proof of lemma \ref{lem:A1}}]
The current $C$ is represented by a couple $(\mu_C, \vec{C})$, where $\mu_C$ is a Radon measure and $\vec{C}$ is a unit $2$-vector field, well defined $\mu_C$-a.e. The $(1,1)$-condition can be expressed by the fact that $\vec{C} = \sum_{j=1}^M \la_j \vec{C}_j$, with $\sum_{j=1}^M \la_j=1$, $\la_j \geq 0$ and $\vec{C}_j$ are unit simple $J_1$-invariant.

The push-forward $f_*C$ can be represented by the Radon measure $f_*\mu_C$ and the $2$-vector field (defined $f_*\mu_C$-a.e.) $f_*\vec{C}$, the latter is however not of unit mass. Denoting by $\|\cdot\|$ the mass norm on $2$-vectors with respect to $g_2$, we rewrite it as 

$$f_*\vec{C} = \sum_{j=1}^M \la_j f_* \vec{C}_j = \sum_{j=1}^M \la_j \cdot \|f_* \vec{C}_j\| \frac{f_* \vec{C}_j}{\|f_* \vec{C}_j\|} = $$

$$=\left(\sum_{j=1}^M \la_j \cdot \|f_* \vec{C}_j\|\right) \sum_{j=1}^M \frac{\la_j \cdot \|f_* \vec{C}_j\|} {\left(\sum_{j=1}^M \la_j \cdot \|f_* \vec{C}_j\|\right)}\frac{f_* \vec{C}_j}{\|f_* \vec{C}_j\|},$$

where each simple $2$-vector $\frac{f_* \vec{C}_j}{\|f_* \vec{C}_j\|}$ is of unit mass and $J_2$-invariant (by the hypothesis on $f$).

We can then represent $f_*C$ by the Radon measure $$\left(\sum_{j=1}^M \la_j \cdot \|f_* \vec{C}_j\|\right) f_* \mu_C$$ and the $2$-vector field of unit mass $$\sum_{j=1}^M \frac{\la_j \cdot \|f_* \vec{C}_j\|} {\left(\sum_{j=1}^M \la_j \cdot \|f_* \vec{C}_j\|\right)}\frac{f_* \vec{C}_j}{\|f_* \vec{C}_j\|},$$ which is a convex combination of unit simple $J_2$-holomorphic $2$-vectors. 

\end{proof}

\medskip

\begin{lem}
\label{lem:A2} [\textbf{the density is preserved}]

Let $U$, $V$ be open sets in $\R^{2n+2}$, $\om$ be a calibration in $U$, $T$ be a normal $\om$-positive $2$-cycle in $U$, $f:U \to V$ be a diffeomorphism. 
Be $\nu(p) \geq 0$ the density of $T$ at $p \in U$. Then the current $f_*T$ has $2$-density equal to $\nu(p)$ at the point $f(p) \in V$.
\end{lem}

\begin{proof}[\bf{proof of lemma \ref{lem:A2}}]
Up to translations, which do not affect densities, we may assume $p=f(p)=0$, the origin of $\R^{2n+2}$. We use coordinates $q=(q_1, q_2, ..., q_{2n+2})$.

\medskip
 
\textit{Step 1}. Assume that $f$ is linear. Choose any sequence of radii $R_n \downarrow 0$ and dilate the current $f_*T$ around $0$ with the chosen factors, i.e. observe the sequence:
\[\left(\frac{Id}{|R_n|}\right)_* \left(f_*T \right) = \left(\frac{Id}{|R_n|} \circ f \right)_*  T  .\]

By the linearity of $f$ this is the same as 
\[ \left(f \circ  \frac{Id}{|R_n|} \right)_*  T  =  f_* \left(\frac{Id}{|R_n|}\right)_* T .\]

The assumptions yield a subsequence $R_{n_j}$ such that $\left(\frac{Id}{|R_{n_j}|}\right)_* T \rightharpoonup T_\infty$ for a cone $T_\infty$, whose density at the vertex is $\nu(0)$. So

\[f_* \left(\frac{Id}{|R_{n_j}|}\right)_* T \rightharpoonup f_*T_\infty .\]

Recall that $T_\infty$ is represented by a positive Radon measure on the $2$-planes, with total mass $\nu(0)$. The linearity of $f$ gives that $f_*T_\infty$ is still a cone with the same density $\nu(0)$ at the vertex, so we have found a subsequence $R_{n_j}$ such that $\left(\frac{Id}{|R_{n_j}|}\right)_* \left(f_*T \right)$ weakly converges to a cone with density $\nu(0)$. Since the sequence $R_n$ was arbitrary, we get in particular that $f_*T$ has $2$-density equal to $\nu(0)$ at the point $f(0)=0$.

\medskip
 
\textit{Step 2}. For a general $f$, write $f(q) = Df(0) \cdot q + o(|q|)$.

As before, we have to observe $\displaystyle \left(\frac{Id}{|R_n|}\right)_* \left(f_*T \right)$. We show that this sequence has the same limiting behaviour as  $\displaystyle \left(\frac{Id}{|R_n|}\right)_* \left((Df(0) \cdot q)_*T \right)$, for which \textit{Step 1} applies.

We estimate the difference of the actions on a two-form $\beta$ supported in the unit ball $B_1$: 

\[\left(\frac{Id}{|R_n|}\right)_*\left[f_*T - (Df(0) \cdot q)_* T\right](\beta) = \]\[=T \left( f^* \left(\frac{Id}{|R_n|}\right)^* \beta - (Df(0) \cdot q)^* \left(\frac{Id}{|R_n|}\right)^* \beta\right).\]  

Writing explicitly $\beta= \sum_I \beta_I dq^I$, where $dq^I = dq^i \wedge dq^j$ for $i \neq j \in \{1, 2, ..., 2n+2\}$, the difference in brackets reads\footnote{Writing $f=(f^1, f^2, ..., f^{2n+2})$ and $I=(i,j)$, the notation $df^I$ stands for $d (f^i) \wedge d (f^j)$, as in \cite{G} (page 120).}

\[\sum_I  \frac{\beta_I \circ \frac{Id}{|R_n|} \circ f - \beta_I \circ \frac{Id}{|R_n|} \circ (Df(0) \cdot q)}{R_n^2} df^I .  \]

 This form is supported, for $n$ large enough, in a ball of radius $\leq \frac{1}{2|Df(0)|}R_n$ around $0$. Moreover, for each $I$, we can estimate from above, for $n$ large enough:

$$|df^I|\left|\frac{\beta_I \circ \frac{Id}{|R_n|} \circ f - \beta_I \circ \frac{Id}{|R_n|} \circ (Df(0) \cdot q)}{R_n^2}\right| \leq $$ $$\leq \frac{\|f\|_{C^1(B_1)}\|\beta_i\|_{C^1(B_1)}}{R_n^3}  \cdot |o(|q|)| \leq \frac{|o(1)|}{R_n^2} , $$

for a function $o(1)$, infinitesimal as $n \to \infty$, depending on $\beta$ and $\|f\|_{C^2}$. Using monotonicity, we get a constant $K>0$, depending on $\nu(0)$ and $\|f\|_{C^1}$, such that $M \left(T \res B_{\frac{1}{2|Df(0)|}R_n} \right)\leq K R_n^2$ for $n$ large enough. These estimates imply

\[T \left( f^* \left(\frac{Id}{|R_n|}\right)^* \beta - (Df(0) \cdot q)^* \left(\frac{Id}{|R_n|}\right)^* \beta\right) \to 0  \text{ as $n \to \infty$, }\]  

so the limiting behaviour of $\displaystyle \left(\frac{Id}{|R_n|}\right)_* \left(f_*T \right)$ must be the same as that of $\displaystyle \left(\frac{Id}{|R_n|}\right)_* \left((Df(0) \cdot q)_*T \right)$. In particular the density of $f_*T$ at the point $f(0)=0$ is  $\nu(0)$.

\end{proof}

\end{document}